\documentclass[a4paper]{amsart}
\usepackage{pstricks}
\usepackage[T1]{fontenc}
\usepackage[utf8]{inputenc}
\usepackage{stmaryrd}
\usepackage{mathrsfs}
\usepackage{amsmath}
\usepackage{amsthm}
\usepackage{amsfonts}
\usepackage{amssymb}
\usepackage{lmodern}
\usepackage{hyperref}

\newtheorem{theorem}{Theorem}
\theoremstyle{plain}

\newtheorem{lemma}[theorem]{Lemma}

\newtheorem{proposition}[theorem]{Proposition}
\newtheorem{remark}{Remark}

\newtheorem*{nontheorem}{Theorem}
\newtheorem*{Ack}{Acknowledgements}

\newcommand{\Ex}{\mathbb{N}}
\newcommand\rightalign[1]{\noalign{\vbox{\hfill $\displaystyle #1$}}}

\newcommand{\Pro}{\mathbf{P}}
\renewcommand{\Pr}{\mathbb P}
\newcommand{\Es}{\mathbb{E}}
\newcommand{\Esp}{\mathbf{E}}
\newcommand{\Espp}{\Es_\infty}
\renewcommand{\phi}{\varphi}
\renewcommand{\epsilon}{\varepsilon}

\newcommand{\EsAr}{\mathbb{T}}

\newcommand{\bra}[2]{\llbracket #1,#2 \rrbracket}

\newcommand{\trac}{\theta(\xrac)}
\newcommand{\xrac}{s_{\emptyset,n}}
\newcommand{\hrac}{h_{\emptyset,n}}

\newcommand{\mas}{\mathbf{m}}
\newcommand{\lf}{\mathsf{x}}

\newcommand{\cb}{{\mathsf B}}

\newcommand{\ce}{{\mathscr E}}

\newcommand{\cf}{{\mathcal F}}
\newcommand{\cg}{{\mathcal G}}

\newcommand{\cn}{{\mathcal N}}

\newcommand{\ct}{{\mathscr T}}

\newcommand{\dt}{\mathsf{T}}

\newcommand{\R}{{\mathbb R}}

\newcommand{\ind}[1]{\mathbf{1}_{#1}}

\newcommand{\expp}[1]{\mathop {\mathrm{e}^{ #1}}}

\begin{document}
 
\title[Fluctuations for the number of records]{Fluctuations for the number of
records on subtrees of the Continuum Random Tree}
\date{\today}
\author{Patrick Hoscheit}
\address{
Patrick Hoscheit, 
Universit\'e Paris-Est, CERMICS,  6-8
av. Blaise Pascal,
  Champs-sur-Marne, 77455 Marne La Vall\'ee, France.
\url{http://cermics.enpc.fr/~hoscheip/home.html}
}
\email{hoscheip@cermics.enpc.fr}

\thanks{This work is partially supported by the French ``Agence Nationale de
  la Recherche'', ANR-08-BLAN-0190.}



\begin{abstract}
We study the asymptotic behavior af the number of cuts $X(T_n)$ needed to
isolate the root in a rooted binary random tree $T_n$ with $n$ leaves. We focus
on the case of subtrees of the Continuum Random Tree generated by uniform
sampling of leaves. We elaborate on a recent result by Abraham and Delmas, who
showed that $X(T_n)/\sqrt{2n}$ converges a.s. towards a Rayleigh-distributed
random variable $\Theta$, which gives a continuous analog to an earlier result
by Janson on conditioned, finite-variance Galton-Watson trees. We prove a
convergence in distribution of $n^{-1/4}(X(T_n)-\sqrt{2n}\Theta)$ towards a
random mixture of Gaussian variables. The proofs use martingale limit theory for
random processes defined on the CRT, related to the theory of records of Poisson
point processes.
\end{abstract}

\maketitle

\section*{Introduction}
The Continuum Random Tree (CRT) is a random metric measure space, introduced by
Aldous (\cite{Aldous1991a, Aldous1993a}) as a scaling limit of various
discrete random tree models. In particular, if we consider $\mu$, a critical
probability measure on $\Ex$, with variance $0<\sigma^2<\infty$ and if we
consider a random Galton-Watson tree $\mathcal{T}_n$ with offspring distribution
$\mu$, conditioned on having $n$ vertices, then we have the following
convergence in
distribution:
\begin{equation} \label{ConvCRT}
\lim_{n\to \infty} \frac{\sigma}{\sqrt n} \mathcal{T}_n = \ct,
\end{equation}
in the sense of Gromov-Hausdorff convergence of compact metric spaces (see for
instance \cite{Duquesne2005a} for more information about the Gromov-Hausdorff
topology), where $\ct$ is a CRT. The family of conditioned Galton-Watson trees
turns out to be quite large, since it contains for instance uniform rooted
planar binary trees (take $\mu(0)=\mu(2)=1/2$) or uniform rooted labelled trees
(Cayley trees, take $\mu(k)=e^1/k!,\ k\ge 0$). There is a combinatorial
characterization of conditioned Galton-Watson trees: they correspond to
the class of so-called \emph{simply generated trees} (see \cite{Janson2012} for
a detailed survey).

In their 1970 paper (\cite{Meir2009}), Meir and Moon considered the problem of
isolating the root through uniform cuts in random Cayley trees. The problem is
as follows: start with a rooted discrete tree $\mathcal{T}_n$, having exactly
$n$ edges (in our context, \emph{rooted} means that, among the $n+1$ vertices of
$\mathcal{T}_n$, one has been distinguished). At each step, remove an edge,
selected uniformly among all edges, then discard the connected component not
containing the root. This procedure is iterated on the remaining tree until the
root is the only remaining vertex. The number $X(\mathcal{T}_n)$ of cuts that is
needed to isolate the root is random, with values in $\llbracket 1,n
\rrbracket$. Meir and Moon showed that when $\mathcal{T}_n$ is a uniform Cayley
tree with $n$ edges, 
\[ \Es[X(\mathcal{T}_n)]\sim \sqrt{\pi n/2}\quad \text{and}\quad
\mathrm{Var}(X(\mathcal{T}_n)) \sim (2-1/\pi)n. \]
 Later, the limiting distribution was found to be the Rayleigh distribution (the
distribution on $[0,\infty)$ with density $x\exp(-x^2/2)dx$) by Panholzer for (a
subset of) the class of simply generated trees (\cite{Panholzer2006}) and, using
a different proof, by Janson for the class of critical, finite-variance,
conditioned Galton-Watson trees (\cite{Janson2006}). 

In \cite{Janson2006}, the distribution of the limiting Rayleigh variable was
obtained using a moment problem, but the question arose whether it had a
connection with the convergence \eqref{ConvCRT} above. Indeed, it is well-known
that the distance from the root to a uniform leaf of the CRT is
Rayleigh-distributed. As a consequence, several approaches were used to describe
a cutting procedure on the CRT that could account for the convergence of
$X(\mathcal{T}_n)/\sqrt{n}$. All these works are relying on the Aldous-Pitman
fragmentation of the CRT, first described in (\cite{Aldous1998}). We will give
below a brief descriptions of this procedure, as it will be central in this
work. Using an extension of the Aldous-Broder algorithm, Addario-Berry, Broutin
and Holmgren described a fragmentation-reconstruction procedure for Cayley trees
and its analog for the CRT. The invariance they prove shows that the limiting
random variable in Janson's result can indeed be realized as the height of a
uniform leaf in a CRT. However, it is not the same CRT as the one arising from
the scaling limit of $\mathcal{T}_n/\sqrt{n}$. Indeed, the random variables
$n^{-1/2}\mathcal{T}_n$ and $n^{-1/2}X(\mathcal{T}_n)$ do \emph{not} converge
jointly to a CRT $\ct$ and the height of a random leaf $H(\ct)$. Bertoin and
Miermont (\cite{Bertoin2012}) describe the so-called \emph{cut-tree}
$\mathrm{cut}(\ct)$ of a given CRT $\ct$ following the genealogy of fragments in
the Aldous-Pitman fragmentation. The limiting variable can then be
described as the height of a uniform leaf in $\mathrm{cut}(\ct)$, which is again
a CRT, thus recovering Rayleigh distribution. \par Following Abraham and Delmas
(\cite{Abraham2011}), we will use a different point of view, based on the
theory of records of Poisson point processes. We will now review some of their
results, in order to set the notations and to describe the framework. 

\subsection*{The Brownian CRT}

In this section, we will recall some basic facts about the Brownian CRT. For
details, see \cite{Aldous1991a,Duquesne2005a,Evans}. We will write $\mathbb{T}$
for the set of (pointed isometry classes of) compact, rooted real trees endowed
with a finite Borel measure. Recall that real trees are metric spaces $(X,d)$
such that 
\begin{itemize}
	\item[(i)] For every $s,t\in X$, there is a unique isometric map
$f_{s,t}$ from $[0,d(s,t)]$ to $X$ such that $f_{s,t}(0)=s$ and
$f_{s,t}(d(s,t))=t$. The image of $f_{s,t}$ is noted $\bra{s}{t}$.
	\item[(ii)] For every $s,t\in X$, if $q$ is a continuous injective map
from $[0,1]$ to $X$ such that $q(0)=s$ and $q(1)=t$,
then $q([0,1])=f_{s,t}([0,d(s,t)])$. 
\end{itemize}
There exists a metric on $\mathbb{T}$ that makes it a Polish metric space, but
we will not attempt to describe it here. For more details, see
\cite{AbrahamGHP,Evans}. 

The Brownian CRT (or Aldous's CRT) is a random element of $\mathbb{T}$, defined
using the so-called \emph{contour process} description: if $f$ is a continuous
nonnegative map $f:[0,\sigma]\to\R_+$, such that $f(0)=f(\sigma)=0$, then the
real tree encoded by $f$ is defined by $\ct_f=[0,\sigma]_{/\sim_f}$,
where $\sim_f$ is the equivalence relation 
\[ 
  x \sim_f y\ \Leftrightarrow\ f(x)=f(y)=\min_{u\in [x\wedge y,x\vee y]}
f(u),\quad x,y\in [0,\sigma].
\]
The metric on $\ct_f$ is defined by 
\[
 d_f(x,y) = f(x)+f(y)-2\min_{u\in [x\wedge y,x\vee y]} f(u),\quad x,y\in
[0,\sigma],
\]
so that $d_f(x,y)=0$ if and only if $x\sim_f y$. Hence, $d_f$ is
definite-positive on $\ct_f$ and defines a true metric. It can be checked (see
\cite{Duquesne2005a}) that $(\ct_f,d_f)$ is indeed a real tree. We define the
\emph{mass-measure} $\mathbf{m}^{\ct_f}(ds)$ on $\ct_f$ as the image of
Lebesgue measure on $[0,\sigma]$ by the canonical projection $[0,\sigma]
\rightarrow \ct_f$. Thus, $\mas^{\ct_f}$ is a finite measure on
$\ct_f$, with total mass $\mas^{\ct_f}(\ct_f)=\sigma$. When the context is
clear, we will usually drop the reference to the tree and write $\mas$ for the
mass-measure $\mas^\ct$.

Now, the Brownian Continuum Random Tree (CRT) corresponds to the real tree
encoded by $f=2B^{\text{ex}}$, twice the normalized Brownian excursion. Since
the length of the normalized Brownian excursion is 1 a.s., the CRT has total
mass 1, \emph{i.e.} the mass measure $\mas$ is a probability measure. The
distribution of the CRT will be noted $\Pr$, or sometimes $\Pr^{(1)}$ if we want
to emphasize the fact that $\mas$ has mass 1. Sometimes, we will consider
\emph{scaled} versions of the CRT. If $r>0$, we consider the scaled Brownian
excursion
\[ 
 B^{\text{ex},r}_t = \sqrt{r} B^{\text{ex}}_{t/r},\ t\in [0,r]
\]
and the associated real tree $\ct_{2B^{\text{ex},r}}$, whose distribution will
be noted $\Pr^{(r)}$. Note that the transformation above corresponds to
rescaling all the distances in a $\Pr^{(1)}$-distributed tree by a factor
$\sqrt{r}$. 

The measure $\mas$ is supported by the set of \emph{leaves} of $\ct$, which are
the points $\mathrm{x}\in \ct$ such that $\ct\setminus \{ \mathrm{x} \}$ is
connected. There is another natural measure $\ell$ defined on the CRT, called
\emph{length measure}, which is $\sigma$-finite and such that $\ell(\llbracket
x,y \rrbracket)=d(x,y)$. Also, the CRT is rooted at one particular vertex
$\emptyset$, which is the equivalence class of $0$, but it can be shown (see
Proposition 4.8 in \cite{Duquesne2005a}) that if $\mathrm{x}$ is chosen
according to $\mas$, then, if $\ct^{\mathrm{x}}$ is the tree $\ct$ re-rooted at
$\mathrm{x}$, $(\ct,\mathrm{x})$ has same distribution as
$(\ct^{\mathrm{x}},\emptyset)$.

When we consider the real tree $\ct$ encoded by $2B$, where $B$ is an
\emph{excursion} of Brownian motion, distributed under the ($\sigma$-finite)
excursion measure $\Ex$, we get that $\ct$ is a compact metric space, with a
length measure $\ell$ and with a finite measure $\mas$. We will write
$\sigma$ for the (random) total mass of $\mas$. Under $\Ex$, $\sigma$ is
distributed as the length of a random excursion of Brownian Motion, that is
\[ \Ex[\sigma \ge t] = \sqrt{\frac{2}{\pi t}}.  \]
The Brownian CRT can be seen as a conditioned version of the tree distributed as
$\Ex[d\ct]$, in the sense that, if $F$ is some nonnegative measurable functional
defined on the tree space $\mathbb{T}$, then
\[ \Ex[F(\ct)] = \int_0^\infty \frac{d\sigma}{\sqrt{2\pi}\ \sigma^{3/2}}
\Es^{(\sigma)}[F(\ct)]. \] 
In the sequel, we will make use of this disintegration of $\Ex$, since some
computations are easier to do under $\Ex$ (see Proposition
\ref{Desintegration}).

\subsection*{The Aldous-Pitman fragmentation}

Given a CRT $\ct$, we consider a Poisson point process \[ \cn(ds,dt) =
\sum_{i\in I} \delta_{(s_i,t_i)}(ds,dt) \]
on $\ct\times \R_+$, with intensity $\ell(ds)\otimes dt$. We will sometimes
refer to $\cn$ as the \emph{fragmentation measure}. If $(s_i,t_i)$ is an atom of
$\cn$, we will say that the point $s_i$ was \emph{marked} at time $t_i$. For
$t\ge 0$, we can consider the connected components of $\ct$ separated by the
atoms of $\cn(\cdot \times [0,t])$. They define a random forest $\mathcal{F}_t$
of subtrees of $\ct$. Aldous and Pitman proved that if we consider the trees
$(\ct_k(t), k\ge 1)$ composing $\mathcal{F}_t$, ranked by decreasing order of
their mass, then the process 
\[
((\mas(\ct_1(t)),\mas(\ct_2(t)),\dots),\ t\ge 0) \]
 is a binary, self-similar fragmentation process, with index 1/2 and erosion
coefficient 0, according to the terminology later framed by Bertoin
(\cite{Bertoin2002}). 

\subsection*{Separation times}

In order to give a continuous analogue to the cutting procedure on discrete
trees described above, we will use the Aldous-Pitman fragmentation on the CRT.
Given a CRT $\ct$ and a fragmentation measure $\cn$, we will define, for any
$s\in \ct$, the \emph{separation time} from the root $\emptyset$ by 
\begin{equation*}
 \theta(s) = \inf\ \{ t\ge 0,\ \cn(\llbracket \emptyset,s \rrbracket \times
[0,t]) \ge 1 \},
\end{equation*}
with the convention $\inf \emptyset = +\infty$. This separation process will be
our main object of study. Note that, under the
definition above, conditionally on $\ct$, $\theta(\emptyset)=\infty$
a.s., and $\theta(s)<\infty$ a.s. for all $s\neq \emptyset$, since $\theta(s)$
is then exponentially distributed with parameter $\ell(\llbracket \emptyset,s
\rrbracket)=d(\emptyset,s)$. Note also that $\theta(s)\to\infty$ when $s\to
\emptyset$, which justifies our convention for $\theta(\emptyset)$.

It is also possible to define the separation process started from any $q\ge 0$,
rather than from infinity. In order to do this, we consider only the marks whose
$t$-component is smaller than $q$:
\begin{equation}
 \theta(s) = \inf\ \{ 0\le t \le q,\ \cn(\llbracket \emptyset,s \rrbracket
\times
[0,t]) \ge 1 \},
\end{equation}
with the convention $\inf \emptyset = q$. Note that, under this definition, we
always have $\theta(\emptyset)=q$, as well as $\lim \theta(s)_{s\to \emptyset} =
q$ a.s. In the case where $q=\infty$, we recover the same distribution as the
separation process defined earlier. The (quenched) distribution of the
separation process started at $q\in [0,\infty]$ on a given CRT $\ct$ will be
noted $\Pr^\ct_q$.

We will also note $\Pr_q^{(r)}$ the (annealed) distribution of the
process $(\theta(s),s\in \ct)$ started at $q\in [0,\infty]$, when $\ct$ is
distributed as a Brownian CRT with mass $r>0$:
\[ 
 \Pr_q^{(r)} = \int_{\EsAr}\Pr^{(r)}(d\ct)\ \Pr_q^{\ct} .
\]
Again, to keep things simple, we will usually work under
$\Pr_\infty=\Pr_\infty^{(1)}$. The jump points of the separation process
correspond to points $s$ marked by the fragmentation measure at a time $t$ where
they belong to the connected component of the root. This implies that they
accumulate in the neighbourhood of the root if $q=\infty$. If $T$ is a subtree
of $\ct$, we note $X(T)$ the number of jumps of the separation process on $T$.
This number can be finite or infinite, according to whether $T$ contains the
root or not, in the case $q=\infty$. 

\subsection*{Linear record process}

One can consider the record process on the real line (\emph{i.e.} when
$\ct=\R_+$), defined using a
Poisson point measure with intensity $ds\otimes dt$. We get, for any $q\in
(0,\infty]$, a random process $(\theta(s),\ s\ge 1)$ such that $\theta(0)=q$,
$\Pr_q^{\R_+}$-a.s. The distribution of this process will be noted
$\Pro_q=\Pr_q^{\R_+}$. We can consider
the jump process 
\[ X_t = \sum_{s\in [0,t]} \ind{\{\theta(s-) > \theta(s)\}}, \]
counting the number of jumps of $\theta$ on $[0,t]$. It should be noted that if
$q=\infty$, then $\theta$ jumps infinitely often in the neighbourhood of the
root, so that a.s. $X_t=\infty$ for any $t>0$. It is easy to check
that, for any bounded, measurable functional $g$ defined on $[0,q]$, we have 
\begin{equation*}
 \Esp_q\left[g(\theta(s))\right] = \expp{-qs}g(q) + \int_0^q g(x)s\expp{-sx}\
dx.
\end{equation*}
In particular, 
\begin{equation}\label{EspTheta}
 \Esp_q\left[\theta(s)\right] = \frac{1-\expp{-qs}}{s}.
\end{equation}
When $q<\infty$, if $t\ge 0$, and conditionally on $\theta(t) = q'$, the next
jump of $\theta$ can be seen to be equal to $\inf\  \{ s\ge t,\
\cn([0,q'],[t,s])
\ge 1 \}$, which
is exponentially distributed, with parameter $q'$. Thus, $X$ is the counting
process of a point measure on $\R_+$ with intensity $\theta(s) ds$. Elementary
properties of counting processes of point measures (see \cite{Abraham2011} for
more details) then show that, for any $q\in (0,\infty)$, the processes 
\begin{gather}\label{MartRec1} \left(N_t = X_t - \int_0^t \theta(s)\ ds,\ t\ge
0\right)  \\
\label{MartRec2} \left(N_t^2-\int_0^t \theta(s)\ ds,\ t\ge
0\right) \\
\label{MartRec3} \left(N_t^4-3 \left(\int_0^t \theta(s)\ ds \right)^2
-\int_0^t\theta(s)\ ds,\ t\ge 0\right)
\end{gather}
are $\Pro_q$-martingales in the natural filtration of $\theta$.

\subsection*{Number of records on subtrees}

Given a CRT $\ct$, let $(\mathrm{x}_n,\ n\ge 1)$ be an iid sequence of leaves of
$\ct$, sampled according to $\mas$. If $n\ge 1$, we consider $\dt_n$, the
subtree spanned by the leaves $(\emptyset,\mathrm{x}_1,\dots,\mathrm{x}_n)$. The
tree $\dt_n$ is a random rooted binary tree with edge-lengths, whose
distribution is explicitly known (see \cite{Aldous1993a}). Its length
$L_n=\ell(\dt_n)$ is known to be distributed according to the
$\mathrm{Chi}(2n)$-distribution, that is 
\begin{equation}\label{Chi2n}
 \Pr(L_n\in dx) = \frac{2^{1-n}}{(n-1)!} x^{2n-1} \exp(-x^2/2)\ind{\{x>0\}}. 
\end{equation}
Note that the case $n=1$ gives a Rayleigh distribution, as was mentioned
earlier. It is proven in \cite{Abraham2011} that, a.s.:
\begin{equation}\label{AsymLong}
 \lim_{n\to\infty} \frac{L_n}{\sqrt{2n}} = 1.
\end{equation}
The tree $\dt_n$ has exactly $2n-1$ edges. The edge adjacent to the root will be
noted $\bra{\emptyset}{\xrac}$, where $\xrac$ is the first branching point in
$\dt_n$; the height of $\xrac$ is noted
$\hrac=\ell(\bra{\emptyset}{\xrac})$. Recall from Proposition 5.3 in
\cite{Abraham2011} that $\sqrt{n}\hrac$ converges in distribution to a
nondegenerate random variable, and that we have the following moment
computation, for $\alpha >-1$:
\begin{equation}\label{eq:MomHn}
 \Es\left[\hrac^{\alpha}\right] = \frac{\Gamma(\alpha+1)}{2^{\alpha/2}}
\frac{\Gamma(n-1/2)}{\Gamma(n+\alpha/2-1/2)} \sim_{n\to\infty}
\Gamma(\alpha+1)2^{-\alpha/2} n^{-\alpha/2}.
\end{equation}
We will also use the notation $\dt_n^*=(\dt_n \setminus
\bra{\emptyset}{\xrac})\cup \{\xrac\}$ for the subtree above the lowest
branching point in $\dt_n$. When a new leaf $\lf_n$ is sampled, it gets attached
to the tree $\dt_{n-1}$ through a new edge, that connects to $\dt_{n-1}$ at the
vertex $s_n\in \dt_{n-1}$. We write 
\[ 
\mathrm{B}_n = (\dt_n\setminus \dt_{n-1})\cup \{s_n\} = \bra{s_n}{\lf_n}.
\]
The quantity $X_n^*$ is the continuum counterpart of the edge-cutting number
$X(T_n)$ that can be found in the literature. Indeed, as soon as a jump appears
on the first edge $\llbracket \emptyset,\xrac \rrbracket$, all subsequent jumps
will be on this edge, even closer to the root. Thus, $X_n^*$ can be seen as the
number of cuts before the first cut on $\llbracket \emptyset,\xrac \rrbracket$
was made. In some sense, the first mark appearing on $\llbracket \emptyset,\xrac
\rrbracket$ is analog to the last cut needed to isolate the root in the discrete
case.

The following theorem is the analog of the convergence (in distribution) that
can be found in \cite{Janson2006} $X(\mathcal{T}_n)/\sqrt{n} \rightarrow
\mathcal{R}$, where $\mathcal{R}$ is Rayleigh-distributed. We will write
$\Theta$ for the mean separation time $\int_\ct \theta(ds) \mas(ds)$.
\begin{nontheorem}[\cite{Abraham2011}]
 We have $\Pr_\infty$-a.s:
\begin{equation}\label{LGN}
 \lim_{n\to\infty} \frac{X_n^*}{\sqrt{2n}} = \Theta.
\end{equation}
Furthermore, under $\Pr_\infty$, $\Theta$ has Rayleigh distribution.
\end{nontheorem}
 Note that $\dt_n^*$ has $2n-2$ edges, so that the rescaling is $\sqrt{2n}$. In
comparison, Janson considers random trees with $n$ edges, which explains the
difference between the two results. It should be noted that Abraham and Delmas
show a slightly more general result, since they consider scaled versions of the
CRT, proving the result under all the measures $\Pr^{(r)}_\infty,\ r>0$. While
our main result, Theorem \ref{The:TCLPrincipal} below is still true in these
cases, we restrict ourselves to the case of Aldous's tree ($r=1$) for
convenience.

The purpose of this work is to investigate the fluctuations of $X_n^*/\sqrt{2n}$
around its limit $\Theta$. It is shown in Theorem \ref{The:TCLPrincipal}, which
is the main result of this work, that these fluctuations are typically of the
order $n^{1/4}$. 

\begin{theorem}\label{The:TCLPrincipal}
  Under $\Pr_\infty$, we have the following convergence in distribution:
\begin{equation}\label{Conv:TCL2}
 \lim_{n\to\infty} n^{1/4} \left( \frac{X_n^*}{\sqrt{2n}} - \Theta \right) = Z,
\end{equation}
where $Z$ is a random variable which is, conditionally on $\Theta$, distributed
according to 
\begin{equation}\label{LaplaceZ} \Espp\left[ \left. e^{i t Z} \right| \Theta
\right] = e^{-t^2 \Theta/\sqrt{2}}.
\end{equation}
\end{theorem}
In other words, $Z$ is distributed as $2^{1/4}\sqrt{\Theta}G$, where $G$ is an
independent standard normal random variable. As $\Theta$ is Rayleigh-distributed
under $\Espp$, the Laplace transform \eqref{LaplaceZ} can be
explicitly computed, but does not correspond to any known distribution. 

The proof of Theorem \ref{The:TCLPrincipal} will be carried out in two steps:
we write 
\begin{equation}\label{Strategie}
\left( \frac{X_n^*}{\sqrt{2n}} - \Theta \right) = \frac{1}{\sqrt{2n}} \left(
X_n^*- \int_{\dt_n^*} \theta(s) \ell(ds) \right) +\left( \frac{1}{\sqrt{2n}}
\int_{\dt_n^*} \theta(s) \ell(ds) -\Theta \right).
\end{equation}
In Section 1, we will show that, when averaging over $\ct$, the variance arising
from the random choice of the leaves $(\mathrm{x}_n,\ n\ge 1)$ does not bring
any significant contribution to \eqref{Conv:TCL2}. We prove this by
decomposing $\ct$ conditionally on its subtree $\dt_n$ and by proving a general
disintegration formula (Lemma \ref{Desintegration}). Therefore, the second term
in \eqref{Strategie} converges to 0 when suitably renormalized.

In Section 2, we prove Theorem \ref{The:TCLPrincipal} by showing that, when
properly rescaled, the difference $(X_n^*-\int_{\dt_n^*} \theta(s) \mas(ds))$ is
asymptotically normally distributed (Proposition \ref{prop:ConvMartFerm}). This
is a consequence of the classical martingale convergence theorems of
\cite{Hall1980}. 

In the Appendix, we collect several technical lemmas. 

\section{Variance in the weak convergence of length measure to mass measure}

The main result of this section is Proposition \ref{prop:Conv1}. 

\begin{proposition}\label{prop:Conv1}
 As $n\to\infty$, we have the following convergence in probability:
\begin{equation} \label{eq:Conv1}
 \lim_{n\to\infty} n^{1/4}\left( \int_{\dt_n^*} \theta(s)
\frac{\ell(ds)}{\sqrt{2n}}
- \Theta \right) = 0.
\end{equation}
\end{proposition}
Recall that, conditionally on $\ct$, we sample independent leaves
$(\mathrm{x}_n,\ n\ge 1)$ with common distribution $\mas(d\mathrm{x})$. We will
consider the filtration $(\cf_n,\ n\ge 1)$ defined by
\begin{equation*}
 \cf_n = \sigma\left(\{ (\mathrm{T}_1,\dots,\mathrm{T}_n),\ (\theta(s),\
s\in \dt_n) \}\right),\quad n\ge 1.
\end{equation*}
A key step in the proof of the a.s. convergence of $X_n^*/\sqrt{2n}$ to $\Theta$
in
\cite{Abraham2011} is the convergence of $M_n=\Espp[\Theta|\cf_n]$. Since
$(M_n,\ n\ge 1)$ is a closed $L^2$ martingale, it converges
$\Pr^{(1)}_\infty$-a.s. (and in $L^2$) towards $M_\infty = \Theta$ (notice that
$\Theta$ is indeed $\cf_\infty$-measurable, since $\cup_{n\ge 1} \dt_n$ is dense
in $\ct$, and since $\theta$ is continuous $\mas$-almost everywhere). The proof
of Proposition \ref{prop:Conv1} will be divided in two. First, we prove the
next proposition: 
\begin{proposition}\label{prop:ConvReste}
  We have the following convergence in probability:
\begin{equation}
 \lim_{n\to\infty} n^{1/4}\left(\frac{1}{\sqrt{2n}}\int_{\dt_n^*} \theta(s)
\ell(ds) - \Espp[\Theta|\cf_n]\right) = 0.
\end{equation}
\end{proposition}
Then, we prove a more precise statement than the convergence of
$\Espp[\Theta|\cf_n]$ towards $\Theta$. 
\begin{proposition}\label{prop:ConvMartFerm}
We have
\begin{equation}\label{ConvMartFerm}
 \lim_{n\to\infty} n^{1/4} \left( \Espp[\Theta |\cf_n] - \Theta
\right) = 0,
\end{equation}
in probability, as $n\rightarrow \infty$.
\end{proposition}
Of course, Propositions \ref{prop:ConvReste} and \ref{prop:ConvMartFerm} imply
Proposition \ref{prop:Conv1}. Before we can prove Proposition
\ref{prop:ConvReste}, we need to describe more precisely how the marked tree
$(\ct,\theta)$ is distributed conditionally on $\cf_n$. 

\subsection{Subtree decomposition}

Given the subtree $\dt_n$, the set $\ct\setminus \dt_n$ is a random forest; let
$(\mathcal{X}_i,\ i\in I_n)$ be the collection of its connected components. For
any connected component $\mathcal{X}_i$ of $\ct\setminus\dt_n$, there is a
unique point $s_i\in \dt_n$ such that
\[ 
 \bigcap_{x\in \mathcal{X}_i} \llbracket \emptyset,x\rrbracket = \llbracket
\emptyset,s_i \rrbracket.
\]
For any $i\in I_n$, we will write $\ct_i$ for the tree $\mathcal{X}_i\cup
\{s_i\}$, rooted at $s_i\in \dt_n$. We will sometimes use
the notation 
\begin{align}\label{def:En}
 \ce_n = & \{ s\in \ct,\ \llbracket \emptyset,s \rrbracket \cap \dt_n^* =
\emptyset \} \\ 
 = & \bigcup_{i\in I_n,\ s_i \in \llbracket
\emptyset,\xrac \rrbracket} \mathcal{X}_i  \nonumber
\end{align}
for the set of all vertices in the tree such that the unique path linking them
to the root intersects $\dt_n$ on $\llbracket \emptyset,\xrac
\rrbracket$. 

Many things are known about the distribution of the forest $(\ct_i,\ i\in I_n)$.
For instance, Pitman pointed out (see \cite{Dong2006}) that the
stickbreaking construction of the CRT in \cite{Aldous1991a} implied that the
sequence $(\mas(\ct_i),\ i\in I_n)$, ranked in decreasing order, is distributed
according to the Poisson-Dirichlet distribution with parameters $\alpha=1/2$ and
$\theta = n-1/2$ (for more background on Poisson-Dirichlet distributions, see
\cite{Pitman2006}). We will give another description, focusing on the tree
structure of $\ct$ conditionally on $\cf_n$. This description can be seen as a
conditioned version of Theorem 3 in \cite{LeGall1993}.

\begin{lemma}\label{Desintegration}
 Let $F$ be a nonnegative functional on $\mathbb{T}\times \dt_n$. Then
\begin{equation}
 \Espp\left[\left.\sum_{i\in I_n} F(\ct_i,s_i) \right|
\cf_n\right]
= \int_0^1 \frac{\expp{-L_n^2v/(2-2v)}}{\sqrt{2\pi} v^{3/2}(1-v)^{3/2}}\ dv
 \int_{\dt_n} \ell(ds)
\Es_{\theta(s)}^{(v)}[F(\ct,s)].
\end{equation}
\end{lemma}
\begin{proof}
 Let $Y$ be a $\cf_n$-measurable random variable; let us compute
$\Es_{\infty}^{(1)}\left[ Y \sum_{i\in I_n}
F(\ct_i,s_i)\right].$
In order to do this computation, we will perform a disintegration with respect
to $\sigma$ in the following expression: for $\mu\ge 0$,
\begin{align*}
 I(\mu) & =  \Ex_\infty\left[Y \sum_{i\in I_n}
F(\ct_i,s_i)\expp{-\mu \sigma} \right] \\
 & =  \Ex_\infty\left[Y \sum_{i\in I_n}
F(\ct_i,s_i)\expp{-\mu \sigma_i} \expp{-\mu\sum_{j\neq i} \sigma_j}
\right]. \\
\intertext{Using a Palm formula, we get:}
 & =  \Ex_\infty\left[ \int_{\dt_n} \ell(ds)
\Ex_{\theta(s)}[F(\ct,s)\expp{-\mu\sigma}] \exp\left(-\int_{\dt_n}\ell(ds)
\int_0^\infty\frac{du}{\sqrt{2\pi}u^{3/2}}(1-\expp{-\mu\sigma})\right) \right]
\\
 & =  \Ex_\infty\left[ Y  \int_{\dt_n} \ell(ds)
 \Ex_{\theta(s)}[F(\ct,s)\expp{-\mu\sigma}] \expp{-L_n\sqrt{2\mu}} \right],
\\
\intertext{since $\Ex[1-\exp(-\mu\sigma)]=\sqrt{2\mu}$. We can disintegrate the
$\sigma$-finite measure
$\Ex_{\theta(s)}$ according to the total mass $\sigma$:}
 I(\mu) & =  \Ex_\infty\left[ Y  \int_{\dt_n} \ell(ds)
\int_0^\infty \frac{dv}{\sqrt{2\pi}v^{3/2}}
\Es_{\theta(s)}^{(v)}[F(\ct,s)\expp{-\mu\sigma}] \expp{-L_n\sqrt{2\mu}}
\right] \\
& =  \Ex_\infty\left[ Y \int_{\dt_n} \ell(ds)
\int_0^\infty
\frac{dv}{\sqrt{2\pi}v^{3/2}}
\Es_{\theta(s)}^{(v)}[F(\ct,s)]\expp{-\mu v} \int_0^\infty L_n
\frac{dr}{\sqrt{2\pi r^3}}
\expp{-\mu r-L_n^2/(2r)} \right],
\end{align*}
using the well-known formula
\[
 \expp{a\sqrt{2s}} = \int_0^\infty \expp{-sr} \frac{a}{\sqrt{2\pi r^3}}
\expp{-a^2/2r}\ dr,
\]
for the Laplace transform of the density
of the 1/2-stable subordinator (see for instance Chapter III, Proposition (3.7)
in \cite{Revuz2004}). By the Fubini-Tonelli theorem, we then get:
\begin{align*}
 I(\mu) & =  \Ex_\infty\left[ Y  \int_{\dt_n} \ell(ds)
\int_0^\infty \frac{dv}{\sqrt{2\pi}v^{3/2}}
\Es_{\theta(s)}^{(v)}[F(\ct,s)]\expp{-\mu v} \int_v^\infty
\frac{L_n\expp{-\mu(t-v)}\ dt}{\sqrt{2\pi}(t-v)^{3/2}}\expp{-L_n^2/(2t-2v)}
\right]\\
 & = \int_0^\infty \frac{\expp{-\mu t}dt}{\sqrt{2\pi}t^{3/2}}
\Ex_\infty\left[
Y  \int_{\dt_n} \ell(ds)
 \int_0^t \frac{L_n t^{3/2}\ dv}{\sqrt{2\pi}v^{3/2}(t-v)^{3/2}}
 \expp{-L_n^2/(2t-2v)}
\Es_{\theta(s)}^{(v)}[F(\ct,s)] \right].
\end{align*}
Now, we can use the scaling property of the marked tree $(\ct,\theta)$ under
$\Ex_\infty$ and that the fact the
total mass $\sigma$ has density $dt/(\sqrt{2\pi}t^{3/2})$ under
$\Ex_\infty$, to get that, for any $\cf_n$-measurable random variable $Y$, 
\begin{equation*}
 \Espp\left[Y \sum_{i\in I_n} F(\ct_i,s_i)\right] 
=  \Ex_\infty\left[Y \frac{1}{\sqrt{2\pi}} \int_0^1 \frac{L_n\
dv}{v^{3/2}(1-v)^{3/2}} \expp{- L_n^2/(2-2v)} \times
\int_{\dt_n}\ell(ds)\Es_{\theta(s)}^{(v)}[F(\ct,s)] \right].
\end{equation*}
Now, recall the absolute continuity relation
the distribution of
$\dt_n$ under $\Ex_\infty$ and under $\Espp$ (Corollary 4 in
\cite{LeGall1993}): for any measurable bounded functional $G$, 
\begin{equation*}
 \Espp[G(\dt_n)] = \Ex_\infty\left[\ell(\dt_n) \expp{-\ell(\dt_n)^2/2} G(\dt_n)
\right].
\end{equation*}
Since $\exp(-L_n^2/(2-2v)) = \exp(-L_n^2/2) \cdot \exp(-L_n^2v/(2-2v))$, we get:
\begin{equation*}
 \Espp\left[Y \sum_{i\in I_n} F(\ct_i,s_i)\right] = \Espp \left[
Y\frac{1}{\sqrt{2\pi}} \int_0^1 \frac{dv}{v^{3/2}(1-v)^{3/2}} \expp{-
L_n^2v/(2-2v)} \int_{\dt_n}\ell(ds)\Es_{\theta(s)}^{(v)}[F(\ct,s)] \right].
\end{equation*}
Taking conditional expectations with respect to $\cf_n$ gives the desired
result.
\end{proof}

\begin{remark}
 Notice that if $F(\ct,s)=\mas(\ct)$, we find the striking identity
\begin{equation}\label{eq:DensMas}
  \frac{1}{\sqrt{2\pi}} \int_0^1
\frac{L_n\expp{-L_n^2v/(2-2v)}}{v^{1/2}(1-v)^{3/2}} dv = 1.
\end{equation}
In other words, the function
$f_a(v)=a\expp{-a^2v/(2-2v)}/(\sqrt{2\pi} v^{1/2}(1-v)^{3/2}))$ is a probability
density on $(0,1)$ for any $a>0$. This probability distribution has already
been described in the context of the Aldous-Pitman fragmentation: if $a>0$,
Aldous and Pitman show that it is the distribution of the size of the fragment
containing the root at time $a$. We refer to \cite{Aldous1998,BertoinRFC} for
more information on the ``tagged fragment'' process in self-similar
fragmentations.
\end{remark}

\subsection{Proof of Proposition \ref{prop:ConvReste}}

We now have everything we need to prove Proposition \ref{prop:ConvReste}.

\begin{proof}[Proof of Proposition \ref{prop:ConvReste}]
We will start from Lemma 7.4 in \cite{Abraham2011}: we have a.s. for $n\ge 1$
\begin{equation}\label{InegVn}
 -R_n \le \Espp[\Theta|\cf_n] - \frac{1}{L_n} \int_{\dt_n^*}
\theta(s)\ \ell(ds) \le V_n,
\end{equation}
where we noted $V_n=\Espp[\int_{\ce_n} \theta(s) \mas(ds)|\cf_n]$ (recall the
definition
of $\ce_n$ in \eqref{def:En}) and where $R_n = \exp(-L_n^2/4)
\theta(h_{\emptyset,n})^2/4$. Furthermore, there $\Pr_\infty$-a.s. exists
a constant $C>0$ such that 
\[
 R_n \le C n^8 \expp{-L_n^2/8}.
\]
Thus, considering that $L_n/\sqrt{2n}$ converges
a.s. to 1 (\ref{AsymLong}), we get that $n^{1/4}R_n$
converges a.s. to 0. Therefore, we needn't worry about the left-hand side of
\eqref{InegVn} and the only thing we need to prove is that $n^{1/4}V_n$
converges in distribution to 0 as $n\to\infty$. The proof in \cite{Abraham2011}
uses a dominated convergence argument to show that $V_n$ a.s. converges to 0,
but we will need a more precise estimate for $V_n$. By definition, using the
notation 
\[
 \Theta_i^{(n)} = \int_{\ct_i} \theta(s)\ \mas(ds),\quad i\in I_n,
\]
we have 
\begin{equation*}
 V_n  = \Espp\left[\int_{\ce_n} \theta(s)\ \mas(ds) \Big|\cf_n\right] = \Espp
\left[\sum_{i\in I_n} \Theta_i \ind{\{ s_i\in \llbracket \emptyset, \xrac
\rrbracket \}} \Big| \cf_n\right].
\end{equation*}
Using the disintegration formula from Lemma \ref{Desintegration}, we get:
\begin{align*}
 V_n & = \frac{1}{\sqrt{2\pi}} \int_0^1
\frac{dv}{v^{3/2}(1-v)^{3/2}}\expp{-L_n^2v/(2-2v)} \int_{\llbracket
\emptyset,\xrac\rrbracket} \Es_{\theta(s)}^{(v)}[\Theta]\ \ell(ds). \\
\intertext{Using the fact that $\theta(s)$ is, conditionally on $\dt_n$,
exponentially distributed with parameter $s$, we get:}
 \Espp\left[V_n|\dt_n\right] & = \frac{1}{\sqrt{2\pi}} \int_0^1
\frac{dv}{v^{3/2}(1-v)^{3/2}}\expp{-L_n^2v/(2-2v)}
\int_0^{\hrac} ds \int_0^\infty s\expp{-st} \Es_{t}^{(v)}[\Theta]\ dt \\
 & \le \frac{1}{2} \int_0^1 \frac{dv}{v^{3/2}(1-v)^{3/2}}\expp{-L_n^2v/(2-2v)}
\int_0^{\hrac} ds \left( \int_0^{v^{-1/2}} stv\expp{-st}\ dt +
\int_{v^{-1/2}}^\infty s\sqrt{v}\expp{-st}\ dt \right),
\end{align*}
using the domination $\Es_q^{(v)}[\Theta] \le
\sqrt{\pi/2}\min(qv,\sqrt{v})$ (Lemma \ref{lem:majoHqr}). For technical
reasons, we will restrict ourselves to the event $\{ \hrac<1/2 \}$, but this
will not be too restrictive, since $\hrac\to0$ a.s. Computing the
integrals, we eventually get that $\Espp[V_n|\dt_n]\ind{\{ \hrac<1/2 \}}$ is
dominated by
\begin{equation*}
 W_n=\left(\frac{1}{2} \int_0^1
\frac{dv}{v^{1/2}(1-v)^{3/2}}\expp{-L_n^2v/(2-2v)}
\int_0^{\hrac} \frac{1-\expp{-s/\sqrt{v}}}{s}\ ds\right)\ind{\{ \hrac<1/2 \}}.
\end{equation*}
We will use the domination $(1-\exp(-s))/s \le \ind{[0,1]}(s) + 2/(s+1)
\ind{(1,\infty)}(s)$, which gives:
\begin{align}
\nonumber W_n & \le \frac{1}{2} \int_0^1 
\frac{\expp{-L_n^2v/(2-2v)}}{v^{1/2}(1-v)^{3/2}}\
dv \left(\frac{\hrac}{\sqrt{v}} \ind{\{\hrac/\sqrt{v} \le 1\}} \right. \\
 & \quad \quad \quad \quad \quad \quad \quad \quad \left. +
\left(1+2\log\left(\frac{\hrac/\sqrt{v}+1}{2}\right)\right)
\ind{\{\hrac/\sqrt{v} \ge 1\}} \right)\ind{\{ \hrac<1/2 \}} \nonumber \\
 &  = \left( \frac{1}{2} \int_0^{\hrac^2}
\frac{\expp{-L_n^2v/(2-2v)}}{v^{1/2}(1-v)^{3/2}} \left(1-2 \log 2 +
2\log\left(\frac{\hrac}{\sqrt{v}}
+1\right)\right)\ dv \right) \ind{\{ \hrac<1/2 \}}
\label{Vn1} \\
  & \quad \quad \quad \quad \quad \quad \quad \quad + \left(\frac{1}{2}
\int_{\hrac^2}^1
\frac{\expp{-L_n^2v/(2-2v)}}{v^{1/2}(1-v)^{3/2}} \frac{\hrac}{\sqrt{v}}\
dv\right) \ind{\{ \hrac<1/2 \}}.
\label{Vn2}
\end{align}
As far as \eqref{Vn1} is concerned, we can dominate $\exp(-\alpha
L_n^2v/(1-v))$ by 1 and $(1-v)^{-3/2}$ by its value at $\hrac^2$, \emph{i.e.}
$(1-\hrac^2)^{-3/2}<(3/4)^{-3/2}$ to get:
\begin{equation*}
 \eqref{Vn1} \le \frac{1}{2(3/4)^{3/2}} \int_0^{\hrac^2}
\frac{dv}{\sqrt{v}}\left(1-2\log2+2\log\left(\frac{\hrac}{\sqrt{v}}
+1\right)\right)\ind{\{ \hrac<1/2 \}} = C\cdot \hrac \ind{\{ \hrac<1/2 \}},
\end{equation*}
where $C$ is some deterministic constant. Concerning \eqref{Vn2}, we can bound
$1/\sqrt{v}$ by $1/\hrac$, to get:
\begin{align*}
 \eqref{Vn2} & \le \left(\frac{1}{L_n} \int_{\hrac^2}^1
\frac{1}{2} \frac{L_n \expp{-L_n^2v/(2-2v)}}{v^{1/2}(1-v)^{3/2}}\ dv\right)
\ind{\{ \hrac<1/2 \}} \\ 
& \le \left(\frac{1}{L_n}\int_{0}^1 \frac{1}{2} \frac{L_n
\expp{-L_n^2v/(2-2v)}}{v^{1/2}(1-v)^{3/2}}\ dv\right)\ind{\{ \hrac<1/2 \}} =
\frac{\sqrt{\pi}}{\sqrt{2}L_n}\ind{\{ \hrac<1/2 \}},
\end{align*}
by equation \eqref{eq:DensMas}. Putting things together, we get that
$\Pr_\infty$-a.s.
\begin{equation}\label{eq:BorneVn}
 \Espp\left[V_n|\dt_n\right]\ind{\{ \hrac<1/2 \}} \le
C\cdot\hrac\ind{\{ \hrac<1/2 \}}+ \frac{\sqrt{\pi}}{\sqrt{2}L_n}. 
\end{equation}
Now, $n^{1/4}\hrac\ind{\{ \hrac<1/2 \}}$ converges in $L^1$ to 0 thanks to
\eqref{eq:MomHn}.
Similarly, an easy moment computation using \eqref{Chi2n} for the density of
$L_n$ shows that $n^{1/4}/L_n$ also converges in $L^1$ to 0, so
that the same is true for $n^{1/4}V_n \ind{\{ \hrac<1/2 \}}$. Hence, $n^{1/4}V_n
\ind{\{ \hrac<1/2 \}}$ converges to 0 in probability. Since a.s. there is a
(random) $n_0\ge1$ such that $\hrac<1/2$ for any $n\ge n_0$, we also get that
$n^{1/4}V_n$ converges to 0 in probability. Combining this with the a.s.
convergence to 0 for $n^{1/4}R_n$, we indeed get a convergence in probability:
\begin{equation}\label{ConvProb}
 \lim_{n\to\infty} n^{1/4}\left(\Espp[\Theta|\cf_n] - \frac{1}{L_n}
\int_{\dt_n^*} \theta(s)\ \ell(ds) \right)=0.
\end{equation}
To get the announced result, we still have to prove that 
\begin{equation}\label{ConvProbLn}
 \lim_{n\to\infty} n^{1/4}\left(\frac{1}{L_n}  - \frac{1}{\sqrt{2n}} \right)
\int_{\dt_n^* }\theta(s)\ \ell(ds) = 0.
\end{equation}
This is not difficult: simply write 
\[ n^{1/4}\left(\frac{1}{L_n} 
-\frac{1}{\sqrt{2n}}\right) \int_{\dt_n^* }\theta(s) \ell(ds)=n^{1/4}\left( 1 -
\frac{L_n}{\sqrt{2n}} \right)\left( \frac{1}{L_n} \int_{\dt_n^*}
\theta(s)\ell(ds) \right).
\]
Now, recall that $1/L_n \int_{\dt_n^*}\theta(s)\ell(ds)$ converges to $\Theta$
$\Pr_\infty$-a.s., hence in probability. Furthermore, we can compute 
\[
 n^{1/2}\Espp\left[\left(1-\frac{L_n}{\sqrt{2n}}\right)^2\right] = n^{1/2}
\Espp\left(1+\frac{L_n^2}{2n}-2\frac{L_n}{\sqrt{2n}}\right),
\]
Using the density \eqref{Chi2n} of $L_n$, we easily get that 
\[
 \Espp[L_n] = \sqrt{2} \frac{\Gamma(n+1/2)}{\Gamma(n)}\quad;
\quad \Espp\left[L_n^2\right]=2n. 
\]
Therefore, after computations, we get $n^{1/2}\Espp[(1-L_n/\sqrt{2n})^2]\sim
1/(8\sqrt{n})$, so that in the end, $n^{1/4}(1-L_n/\sqrt{2n})$ converges to 0 in
$L^2$. This implies convergence in probability, hence the
convergence of \eqref{ConvProbLn}. 
\end{proof}

\subsection{Rate of convergence in the Martingale Convergence Theorem}

Before we can move on to Proposition \ref{prop:ConvMartFerm}, we are going to
state a lemma that will be needed in the proof.

\begin{lemma}\label{Theta3/2}
 If $1<\alpha<2$, then, the sequence $\int_{\dt_n^*}
\theta(s)^\alpha \ell(ds)/L_n$ is bounded in $L^1(\Pr_\infty)$.
\end{lemma}
\begin{proof}
 The main idea is that the measure $\ell(ds)/L_n$ converges a.s. to the mass
measure $\mas(ds)$, in the sense of weak convergence of probability measures
on $\ct$. Since the function $\theta$ is neither continuous nor bounded on
$\ct$, we cannot use this fact directly, but it will be the inspiration for
the proof. We will compute the first moment of $Z_n=\int_{\dt_n^*}
\theta(s)^\alpha
\ell(ds)/L_n^*$, using the notation $L_n^*=\ell(\dt_n^*)$. Since
$\theta(s)$ is, conditionally on $\ct$, exponentially distributed with parameter
$\ell(\llbracket \emptyset,s\rrbracket)$, we get
\begin{align}
 \Espp[Z_n] & = \Espp\left[ \int_{\dt_n^*} \ell(\llbracket
\emptyset,s\rrbracket)^{-\alpha}\ \frac{\ell(ds)}{L_n^*} \right] \nonumber \\
 & = \Espp \left[ \int_{\ct} (d(\emptyset,s) -
d(s,\dt_n^*))^{-\alpha}\ind{\ct\setminus \mathcal{E}_n}(s)\ \mas(ds) \right],
\nonumber \end{align}
where $d(s,\dt_n^*)$ is the distance from the leaf $s$ to the closed
subtree $\dt_n^*$ of $\ct$. The last equality comes from the fact that if $s$ is
a leaf of $\ct$ selected uniformly (according to $\mas(ds)$) among all leaves of
$\ct\setminus\mathcal{E}_n$, then its projection $\pi(s,\dt_n)$ on
$\dt_n$ is uniformly distributed (according to length measure) on $\dt_n^*$. We
will rewrite the last expression so as to make the leaves
$\emptyset,\lf_1,\dots,\lf_n$ apparent. The
set $\ct\setminus\mathcal{E}_n$ can be written as 
\begin{equation}\label{CompEn}
 \ct\setminus\mathcal{E}_n = \{s\in \ct,\ \llbracket
\emptyset,\pi(\emptyset,\dt_n^*) \rrbracket \cap \llbracket
s,\pi(s,\dt_n^*) \rrbracket =\emptyset\},
\end{equation}
since $\pi(\emptyset,\dt_n^*)=\xrac$. Note that $\dt_n^*$ is actually the
subtree spanned by the $n$ leaves $\lf_1,\dots,\lf_n$ and that its definition
does not depend on $\emptyset$ or on $s$.

We then apply the fundamental \emph{re-rooting invariance} of the
Brownian CRT, which implies, in this context, that when re-rooting $\ct$ at
$s$, the re-rooted tree $\ct^s$ is distributed as a CRT, and the sequence
$(\emptyset,\lf_1,\dots,\lf_n)$ is distributed as a sample of $n+1$ uniform
leaves in $\ct^s$. Thus, 
\begin{align*}
 \Espp[Z_n] & = \Espp \left[ \int_{\ct} (d(\emptyset,s) -
d(s,\dt_n^*))^{-\alpha}\ind{\{\llbracket
\emptyset,\pi(\emptyset,\dt_n^*) \rrbracket \cap \llbracket
s,\pi(s,\dt_n^*) \rrbracket =\emptyset\}}(s)\ \mas(ds) \right] \\
 & = \Espp\left[\int_{\ct}(d(\emptyset,s) -
\hrac)^{-\alpha}\ind{\{\llbracket
\emptyset,\pi(\emptyset,\dt_n^*) \rrbracket \cap \llbracket
s,\pi(s,\dt_n^*) \rrbracket =\emptyset\}}(s)\ \mas(ds) \right], \nonumber \\
\end{align*}
since in the re-rooting, $d(s,\dt_n^*)$ becomes
$d(\emptyset,\dt_n^*)=h_{\emptyset,n}$. Therefore, we get, using \eqref{CompEn}
again,
\begin{align*}
 \Espp[Z_n] & =  \Espp\left[\int_{\ct}(d(\emptyset,s) -
\hrac)^{-\alpha}\ind{\ct\setminus \mathcal{E}_n}(s)\ \mas(ds) \right] \\ 
& = \Espp\left[ \int_{\ct_n^{(1)} \cup \ct_n^{(2)}} d(\xrac,s)^{-\alpha}
\mas(ds) \right], 
\end{align*}
where $\ct_n^{(1)}$ and $\ct_n^{(2)}$ are the
connected components of $\ct\setminus(\ce_n\cup\{\xrac\})$, joined together by
their common root $\xrac$. We can now use the self-similarity property of the
fragmentation at heights of the Brownian CRT (see \cite{Bertoin2002}) which
shows that, conditionally on $\sigma_n^{(1)}=\mas(\ct_n^{(1)})$ and
$\sigma_n^{(2)}=\mas(\ct_n^{(2)})$, the trees $\ct_n^{(1)}$ and $\ct_n^{(2)}$
are rescaled copies of the Brownian CRT. Thus, 
\begin{align*}
 \Espp[Z_n] & = \Espp\left[ \int_{\ct_n^{(1)}}
d(\xrac,s)^{-\alpha} \mas(ds) \right] 
+ \Espp\left[ \int_{\ct_n^{(2)}} d(\xrac,s)^{-\alpha} \mas(ds) \right] \\
 & = \Espp\left[ \int_{\ct_n^{(1)}} (\sigma_n^{(1)})^{-\alpha/2}
\left(\frac{d(\xrac,s)}{(\sigma_n^{(1)})^{1/2}}\right)^{-\alpha}
\ \sigma_n^{(1)} \frac{\mas(ds)}{\sigma_n^{(1)}} \right] \\
 & \quad + \Espp\left[ \int_{\ct_n^{(2)}} (\sigma_n^{(2)})^{-\alpha/2}
\left(\frac{d(\xrac,s)}{(\sigma_n^{(2)})^{1/2}}\right)^{-\alpha}
\ \sigma_n^{(2)} \frac{\mas(ds)}{\sigma_n^{(2)}} \right] \\
 & = \Espp\left[
\left((\sigma_n^{(1)})^{1-\alpha/2}+(\sigma_n^{(2)})^{1-\alpha/2} \right)
\int_{\ct} d(\emptyset,s)^{-\alpha} \mas(ds) \right],
\end{align*}
using the scaling invariance of the Brownian CRT. Then, as
$0<1-\alpha/2$, we can simply dominate $(\sigma_n^{(1)})^{1-\alpha/2}$ and
$(\sigma_n^{(1)})^{1-\alpha/2}$ by 1 to get that
\begin{equation*}
\Espp[Z_n] \le 2\cdot \Espp\left[\int_\ct
d(\emptyset,s)^{-\alpha}\mas(ds)\right]. 
\end{equation*}
Now, since $d(\emptyset,s)$ is Rayleigh-distributed
under $\Espp$, we easily see that it has moments of order $-\alpha$ for any
$\alpha<2$, which shows that $\Espp[Z_n]$ is indeed bounded, ending our proof.
\end{proof}

We can now turn to the proof of Proposition \ref{prop:ConvMartFerm}. 

\begin{proof}[Proof of Proposition \ref{prop:ConvMartFerm}] Let
$M_n=\Espp[\Theta|\cf_n]$. We will use the fact that 
\begin{equation}\label{6*}
 n^{1/4}\left(\Theta-\Espp[\Theta|\cf_n]\right) = n^{1/4} \sum_{k=n+1}^\infty
\Espp[M_k-M_{k-1}|\cf_{k-1}].
\end{equation}
Let $\epsilon>0$ be small enough, and consider the events
\[
 E_k^1 = \{ L_{k} \ge k^{1/2-\epsilon} \}\quad;\quad E_k^2 = \{
k^{-2} \le h_{\emptyset,k} \le 1/2 \},\quad k\ge 1.
\]
Recalling that $L_k^2$ is distributed as the sum of $k$ independent exponential
random variables with parameter 1, a simple application of Chernoff's inequality
shows that 
\begin{equation}\label{Ek1}
 \Pr_\infty\left(E_k^1\right) \ge 1-k^{-\epsilon k}.
\end{equation}
For $E_k^2$, we can use the moment estimation \eqref{eq:MomHn} for
$h_{\emptyset,k}$ to find that, for any $0\le \alpha\le 1$, and for any
$\beta>0$, 
\begin{align*} 
 1-\Pr_\infty\left(E_k^2\right) & = \Pr_\infty\left( \{ h_{\emptyset,k} >1/2 \}
\cup \{ h_{\emptyset,k} < k^{-2} \} \right) \\
 & \le \Pr_\infty(h_{\emptyset,k}>1/2) + \Pr_\infty\left( h_{\emptyset,k}^{-1}
\ge k^2\right) \\
 & \le 2^\beta \Espp\left[h_{\emptyset,k}^\beta\right] + k^{-2\alpha}
\Espp\left[h_{\emptyset,k}^{-\alpha}\right] \\
 & \sim C\cdot k^{-\beta/2} + C'\cdot k^{-2\alpha} k^{\alpha/2}.
\end{align*}
Hence, by taking $\alpha =1-\eta$ (and $\beta>3\alpha$), we get, for
sufficiently large $k$,
\begin{equation}\label{Ek2}
 \Pr_\infty\left(E_k^2\right) \ge 1-k^{-3/2+3/2\eta}.
\end{equation}
Thus, combining equations \eqref{Ek1} and \eqref{Ek2}, we get that 
\begin{equation*}
 \sum_{k\ge 1} \Pr_\infty\left( \left(E_k^1\right)^c \cup \left(E_k^2\right)^c
\right) \le \sum_{k\ge 1}
\Pr_\infty\left((E_k^1)^c\right) +
\Pr_\infty\left((E_k^2)^c\right) < \infty.
\end{equation*}
Thus, by the Borel-Cantelli lemma, there a.s. exists  $k_0\ge 1$ such that for
$k\ge k_0$, $L_k \ge  k^{1/2-\epsilon}$ and $k^{-2} \le h_{\emptyset,k} \le
1/2$. We will use this truncating events in the following way: since the event
$E_k=E_k^1\cap E_k^2$ is $\cf_k$-measurable, the usual martingale computations
show that 
\begin{equation*}
 \Espp\left[ \left( n^{1/4} \sum_{k=n}^\infty (M_k-M_{k-1}) \ind{E_{k-1}}
\right)^2 \right]  = n^{1/2} \sum_{k=n}^\infty \Espp[ (M_k-M_{k-1})^2
\ind{E_{k-1}}].
\end{equation*}
We will now give precise estimations of $\Espp[(M_k-M_{k-1})^2|\cf_{k-1}]$
using the disintegration formula from Lemma \ref{Desintegration}. By
definition, for all $k\ge 1$, we can write 
\[ \Theta=\int_\ct
\theta(s)\ \mas(ds)=\sum_{i\in I_k} \Theta_i^{(k)}. \]
Then, 
\begin{align*}
 M_k & = \Espp[\Theta|\cf_k] \\
 & = \Espp\Bigg[ \Big. \sum_{i\in I_{k-1}} \Theta_i^{(k-1)} \Big|
\cf_k\Bigg] \\
 & = \Espp[ \Theta_{i_k} |\cf_k] + \Espp\Bigg[\sum_{i\in
I_{k-1}\setminus \{i_k\}} \Theta_i \Big|\cf_k\Bigg],
\end{align*}
where $i_k$ is the unique index in $I_{k-1}$ such that $\mathrm{x}_k\in
\ct_{i_k}$. We then define:
\begin{align*}
 G_k & = \Espp\left[\Theta_{i_k}^{(k-1)}|\cf_k\right] \\
 H_k & = \Espp\Bigg[\sum_{i\in I_{k-1}\setminus \{i_k\}}
\Theta_i^{(k-1)}\Big|\cf_k\Bigg] -
\Espp\Bigg[ \sum_{i\in I_{k-1}} \Theta_i^{(k-1)} \Big|\cf_{k-1}\Bigg],
\end{align*} 
so that we have $  M_k-M_{k-1} = G_k + H_k$ and 
\begin{equation*}
 \Espp\left[(M_k-M_{k-1})^2 | \cf_{k-1}\right] \le 2
\Espp\left[G_k^2|\cf_{k-1}\right] +
2\Espp\left[H_k^2|\cf_{k-1}\right].
\end{equation*}
As far as $G_k$ is concerned, conditionally on $\cf_k$,
$\Theta_{i_k}^{(k-1)}$ can be written as
$\sum_{i\in I_k} \Theta_i^{(k)} \ind{\{s_i\in \cb_k\}}$, so
that we can use the disintegration formula of Lemma \ref{Desintegration} to
get:
\begin{equation*}
 G_k = \frac{1}{\sqrt{2\pi}} \int_0^1
\frac{\expp{-L_k^2v/(2-2v)}}{v^{3/2}(1-v)^{3/2}}\ dv \int_{\cb_k}
\Es_{\theta(s)}^{(v)}[\Theta]\ \ell(ds).
\end{equation*}
Hence, using this expression, we can now compute:
\begin{align*}
 \Espp[G_k^2 | \cf_{k-1}] & =\Espp\left[ \left(\frac{1}{\sqrt{2\pi}} \int_0^1
\frac{\expp{-L_k^2v/(2-2v)}}{v^{3/2}(1-v)^{3/2}}\ dv \int_{\cb_k}
\Es_{\theta(s)}^{(v)}[\Theta]\ \ell(ds)\right)^2\Big|\cf_{k-1}\right] \\
 & \le \Espp\left[\left(\frac{1}{\sqrt{2\pi}} \int_0^1
\frac{\expp{-L_k^2v/(2-2v)}}{v^{1/2}(1-v)^{3/2}}\ dv \int_{\cb_k} \ell(ds)
\theta(s)\right)^2\Big|\cf_{k-1}\right],
\end{align*}
since $\Es_{\theta(s)}^{(v)}[\Theta]\le v\theta(s)$ (Lemma \ref{lem:majoHqr}).
Now, the measure
\[  \frac{L_n\expp{-L_n^2v/(2-2v)}}{\sqrt{2\pi}v^{1/2}(1-v)^{3/2}}dv \]
is a probability density on [0,1] (cf. \eqref{eq:DensMas}), so that we get,
using the fact that $L_{k-1}<L_k$, 
\begin{align*}
 \Espp\left[G_k^2 | \cf_{k-1}\right] & \le \Espp\left[
\left(\frac{1}{L_k}
\int_{\cb_k}
\ell(ds) \theta(s)\right)^2\Big|\cf_{k-1}\right] \\
 & \le \frac{1}{L_{k-1}^2} \Espp\left[ \left(\int_{\cb_k}
\ell(ds) \theta(s)\right)^2\Big|\cf_{k-1}\right].
\end{align*}

Now, conditionally on $\cf_{k-1}$, the record process on $\cb_k$ has the
distribution of an independent record process on $\R_+$, started from
$\theta(s_k)$, stopped at time $\ell(\cb_k)$. Furthermore, it is a
consequence from the stickbreaking construction of Aldous (see
\cite{Aldous1991a}) that, conditionally on $\cf_{k-1}$, the random variables
$s_k$ and $\ell(\cb_k)$ are independent. Furthermore, $s_k$ is distributed
uniformly on $\dt_{k-1}$, and $\ell(\cb_k)$ can be expressed as the length of
the interval between the $(k-1)$th and the $k$th jump of a Poisson process
with intensity $t \ind{[0,\infty)}(t)dt$. Therefore, conditionally on
$\cf_{k-1}$, $\ell(\cb_k)$ has density
\begin{equation} \label{eq:DensLBk}
 r_{L_{k-1}}(dx) = (L_{k-1}+x) \expp{-x^2/2-L_{k-1}x}\ dx.
\end{equation}
Thus, using the notation $F(q,t)= \Esp_q[(\int_0^t \theta(s) ds)^2]$ for
$0<q<\infty$ and $t\ge 0$, we get
\begin{equation*}
 \Espp\left[G_k^2|\cf_{k-1}\right] \le \frac{1}{L_{k-1}^2} \int_{\dt_{k-1}}
\frac{\ell(ds)}{L_{k-1}} \int_0^\infty r_{L_{k-1}}(dx) F(\theta(s),x). 
\end{equation*}
We will cut the integral in two parts, according
to  $\dt_{k-1} = \dt_{k-1}^* \cup(\dt_{k-1} \setminus
\dt_{k-1}^*)$. We then use Lemma \ref{lem:Fql} to dominate $F(\theta(s),x)$:
inequality \eqref{Fql1} for $s\in \dt_{k-1}^*$ and \eqref{Fql2} for $s\in
\dt_{k-1} \setminus \dt_{k-1}^*$. This leads to:
\begin{align*}
 \Espp\left[G_k^2|\cf_{k-1}\right] & \le \frac{1}{L_{k-1}^2}
\int_{\dt_{k-1}^*}
 \frac{\ell(ds)}{L_{k-1}} \int_0^\infty r_{L_{k-1}}(dx) \left(C_1
\theta(s)^{3/2}x^{3/2} + C_2 \theta(s) x^2\right) \nonumber\\
\rightalign{+\frac{1}{L_{k-1}^2} \int_{\dt_{k-1} \setminus \dt_{k-1}^*}
 \frac{\ell(ds)}{L_{k-1}} \int_0^\infty r_{L_{k-1}}(dx) \left(C_3
\theta(s)^{1/2}x^{1/2} + C_4 \theta(s)^{-1/2} x^{1/2}\right)} 
& = \frac{C_1}{L_{k-1}^2} \left(\int_0^\infty r_{L_{k-1}}(dx)
x^{3/2}\right)\left(\int_{\dt_{k-1}^*} \frac{\ell(ds)}{L_{k-1}}
\theta(s)^{3/2}\right) 
\\
& + \frac{C_2}{L_{k-1}^2} \left(\int_0^\infty r_{L_{k-1}}(dx)
x^{2}\right)\left(\int_{\dt_{k-1}^*} \frac{\ell(ds)}{L_{k-1}} \theta(s)\right)
\\
& + \frac{C_3}{L_{k-1}^3} \left(\int_0^\infty r_{L_{k-1}}(dx)
x^{1/2}\right)\left(\int_{\dt_{k-1} \setminus \dt_{k-1}^*}
\ell(ds) \theta(s)^{1/2}\right)\\
& + \frac{C_4}{L_{k-1}^3} \left(\int_0^\infty r_{L_{k-1}}(dx)
x^{1/2}\right)\left(\int_{\dt_{k-1} \setminus \dt_{k-1}^*}
\ell(ds) \theta(s)^{-1/2}\right).
\end{align*}
We can then compute, using Lemma \ref{Momentsr} for the asymptotic moments of
$r_{L_{k-1}}(dx)$: 
\begin{align}
 \Espp\left[ G_k^2 \ind{E_{k-1}} \right] & =
\Espp\left[\Espp\left[G_k^2|\cf_{k-1}\right] \ind{E_{k-1}}\right]
\nonumber  \\ 
& \le \Espp\left[ \int_{\dt_{k-1}^*} \frac{\ell(ds)}{L_{k-1}} \theta(s)^{3/2}
\right] \cdot O(k^{-7/4 + 7/2\epsilon}) \label{Ak1} \\ 
& + \Espp\left[\left(\int_{\dt_{k-1}^*} \frac{\ell(ds)}{L_{k-1}}
\theta(s)\right) \ind{E_{k-1}} \right] \cdot O(k^{-2+4\epsilon}) \label{Ak2} \\
& + \Espp\left[\left(\int_{\dt_{k-1} \setminus \dt_{k-1}^*}
\ell(ds) \theta(s)^{1/2}\right) \right]\cdot
O(k^{-7/4+7/2\epsilon}) \label{Ak3}\\
& + \Espp\left[\left(\int_{\dt_{k-1} \setminus \dt_{k-1}^*}
\ell(ds) \theta(s)^{-1/2}\right) \right] \cdot
O(k^{-7/4+7/2\epsilon}).\label{Ak4}
\end{align}
Using Lemma \ref{Theta3/2}, we see that \eqref{Ak1} is indeed of the order
$k^{-7/4+7/2\epsilon}$. As far as \eqref{Ak2} is concerned, we will show the
following lemma, which will be useful later on, and
which implies in particular that \eqref{Ak2} is of the order
$k^{-2+4\epsilon}$.

\begin{lemma}\label{lem:BorneL2}
 $\Espp\left[\left(\int_{\dt_{k-1}^*} \theta(s)
\ell(ds)/L_{k-1}\right)^2\ind{E_{k-1}}\right]$ is bounded as $k\to \infty$. 
\end{lemma}

\begin{proof}[Proof of Lemma \ref{lem:BorneL2}] Recall \eqref{InegVn}:
\begin{equation}
 -R_{k-1} \le \Espp[\Theta|\cf_{k-1}] - \frac{1}{L_{k-1}} \int_{\dt_{k-1}^*}
\theta(s)\ \ell(ds) \le V_{k-1}. 
\end{equation}
Therefore, we can write
\begin{align*}
 \Espp\left[\left(\int_{\dt_{k-1}^*} \frac{\ell(ds)}{L_{k-1}}
\theta(s)- \Espp[\Theta|\cf_{k-1}]\right)^2 \ind{E_{k-1}} \right] & \le
\Espp\left[ (R_{k-1} \vee V_{k-1})^2\ind{E_{k-1}}\right] \\
&\le \Espp\left[R_{k-1}^2 \ind{E_{k-1}}\right] + \Espp\left[V_{k-1}^2
\ind{E_{k-1}}\right].
\end{align*}
Using \eqref{eq:BorneVn}, we can see that, since $E_{k-1}\in \sigma(\{\dt_n\})$,
\[
 \Espp\left[V_{k-1}^2\ind{E_{k-1}}\right] \le
\Espp\left[\left(C\cdot h_{\emptyset,k-1} +\frac{\sqrt{\pi}}{\sqrt{2}L_{k-1}}
\right)^2\ind{E_{k-1}}\right] \le \Espp\left[(C \cdot h_{\emptyset,k-1} +
\sqrt{\pi/2}/L_{k-1})^2\right].
\]
Hence, as $h_{\emptyset,k-1}$ and $L_{k-1}^{-1}$ are integrable and
decrease to
0 a.s., $\Espp[V_{k-1}^2\ind{E_{k-1}}]$ converges to 0 by monotone convergence.
As for
$\Espp[R_{k-1}^2 \ind{E_{k-1}}]$, we use the fact that, conditionally on
$\dt_{k-1}$, $\theta(h_{\emptyset,{k-1}})$ is exponentially distributed with
parameter $h_{\emptyset,{k-1}}$ to find
\begin{align*}
 \Espp\left[ R_{k-1}^2 \ind{E_{k-1}}\right] & = \Espp\left[ \frac{1}{16}
\expp{-L_{k-1}^2/2} h_{\emptyset,k-1}^{-4} \ind{E_{k-1}}\right] \\
 & \le \frac{1}{16} k^8 \Espp\left[\expp{-L_{k-1}^2/2}\right],
\end{align*}
which easily converges to 0 as $k\to\infty$. Hence, since
$\Espp[\Theta|\cf_{k-1}]$ converges in $L^2$ to $\Theta$, it is of course
$L^2$-bounded, so that $\Espp[(\int_{\dt_{k-1}^*} \theta(s)
\ell(ds)/L_{k-1})^2\ind{E_{k-1}}]$ is indeed bounded as $k\to\infty$, as
announced. \end{proof}

In the two remaining terms \eqref{Ak3} and \eqref{Ak4}, the integral
is taken on a single branch; therefore, we can use the linear case to get
\begin{align*}
 \Espp\left[\left(\int_{\dt_{k-1} \setminus \dt_{k-1}^*}
\ell(ds) \theta(s)^{1/2}\right) \ind{E_{k-1}} \right] & = \Espp\left[
 \Esp_\infty\left[ \int_0^{h_{\emptyset,k-1}} \theta(s)^{1/2} ds
\right] \ind{E_{k-1}}\right] \\
 & = C\cdot \Espp[ h_{\emptyset,k-1}^{1/2} \ind{E_{k-1}} ],
\end{align*}
which easily converges to 0 as $k\to \infty$. A similar argument shows that
\eqref{Ak4} converges to 0 as $\Espp[h_{\emptyset,k-1}^{3/2}\ind{E_{k-1}}]$.
Putting everything together, we find that $\Espp[G_k^2 \ind{E_{k-1}}]$ is of
the order $k^{-7/4+7/2\epsilon}$ as $k\to \infty$, so that the remainder
$\sum_{k=n}^\infty \Espp[G_k^2 \ind{E_{k-1}}]$ is of the order
$n^{-3/4+7/2\epsilon}$. 

Turning to $H_k$, we note that $I_{k-1}\setminus \{i_k\}= \{ i\in I_k,\
s_i\notin \cb_k\}$, so that, using Lemma
\ref{Desintegration}, we get:
\begin{align*}
 H_k = & \Espp\left[\sum_{i\in I_{k}}
\Theta_i^{(k)} \ind{\{s_i\notin \cb_k\}}\Bigg|\cf_k\right]
- \Espp\Bigg[ \sum_{i\in I_{k-1}} \Theta_i^{(k-1)} \Bigg|\cf_{k-1}\Bigg] \\
 = & \int_0^1 \frac{\expp{-L_k^2v/(2-2v)}}{\sqrt{2\pi} v^{3/2}(1-v)^{3/2}}\ dv
 \int_{\dt_k} \ell(ds) \Es_{\theta(s)}^{(v)}[\Theta]\ind{\{s\notin \cb_k\}} \\
 \rightalign{ - \int_0^1 \frac{\expp{-L_{k-1}^2v/(2-2v)}}{\sqrt{2\pi}
v^{3/2}(1-v)^{3/2}}\
dv  \int_{\dt_{k-1}} \ell(ds) \Es_{\theta(s)}^{(v)}[\Theta],} \end{align*}
thus, considering that $\dt_k=\dt_{k-1} \cup (\cb_k\setminus \{s_k\})$, and
that of course $\ell(\{s_k\})=0$,
\begin{align*}
H_k = & \int_0^1 \frac{dv}{\sqrt{2\pi}
v^{3/2}(1-v)^{3/2}}\int_{\dt_{k-1}}
\ell(ds) \Es_{\theta(s)}^{(v)}[\Theta]
\left(\expp{-L_k^2v/(2-2v)}-\expp{-L_{k-1}^2v/(2-2v)}\right) .
\end{align*}
We then use the inequality $|\expp{-a t}-\expp{-a s}|\le a
\expp{-a t}(s-t)$, valid for any $a >0$, and $t\le s$, to find:
\begin{multline}\label{eq:domBk}
 \Espp\left[H_k^2 |\cf_{k-1}\right] \le \left(\frac{1}{\sqrt{2\pi}} \int_0^1
\frac{\expp{-L_{k-1}^2v/(2-2v)}}{v^{3/2}(1-v)^{3/2}}\frac{v}{2-2v}\
dv\int_{\dt_{k-1}}  \Es_{\theta(s)}^{(v)}[\Theta]\ \ell(ds)
\right)^2 \\
\times\Espp\left[(L_k^2-L_{k-1}^2)^2|\cf_{k-1}\right].
\end{multline}
On the one hand, we will use the change of variables
\[ u= L_{k-1}^2 v/(2-2v) \Leftrightarrow v=u/(L_{k-1}^2/2+u). \]
in the integral, which gives:
\begin{equation}\label{BCurs}
 \left(\frac{1}{\sqrt{2\pi}} \int_0^\infty 
\frac{\expp{-u}}{(L_{k-1}^2/2)^{1/2}}
\frac{L_{k-1}^2/2+u}{\sqrt{u}}\ du  \int_{\dt_{k-1}}
\frac{\ell(ds)}{L_{k-1}^2} \Es_{\theta(s)}^{(u/(L_{k-1}^2/2+u))}[\Theta]
\right)^2.
 \end{equation}
We then cut the integral in two parts, according to $\dt_{k-1}=\dt_{k-1}^* \cup
(\dt_{k-1}\setminus \dt_{k-1}^*)$, and we use the simple domination
$\Es_{\theta(s)}^{(v)}[\Theta] \le v \theta(s)$ on $\dt_{k-1}^*$, and the
domination $\Es_{\theta(s)}^{(v)}[\Theta] \le \Es_\infty^{(v)}[\Theta] =
\sqrt{\pi v/2}$ on $\dt_{k-1}\setminus \dt_{k-1}^*$ to get
\begin{multline*}
\eqref{BCurs} \le \left(\frac{1}{\sqrt{\pi}} \int_0^\infty  \frac{du}{L_{k-1}^3}
\frac{ L_{k-1}^2/2+u}{\sqrt{u}} \expp{-u} \int_{\dt_{k-1}^*} \ell(ds)
\theta(s) \frac{u}{ L_{k-1}^2/2+u} \right. \\ \left. + \int_0^\infty 
\frac{du}{L_{k-1}^3} \frac{L_{k-1}^2/2+u}{\sqrt{2u}} \expp{-u} h_{\emptyset,k-1}
\frac{\sqrt{u}}{\sqrt{L_{k-1}^2/2+u}}\right)^2.
\end{multline*}
The integrals can be computed, giving
\[ \eqref{BCurs} \le \left(\frac{1}{2\sqrt{\pi}L_{k-1}^2}
\int_0^\infty \sqrt{u} \expp{-u} \int_{\dt_{k-1}^*} \frac{\ell(ds)}{L_{k-1}}
\theta(s) + \int_0^\infty
\frac{du}{\sqrt{2}L_{k-1}^2}\sqrt{1/2+u/L_{k-1}^2}\expp{-u}
h_{\emptyset,k-1}\right)^2. \]
On the other hand, the term $\Espp[(L_k^2-L_{k-1}^2)^2|\cf_{k-1}]$ appearing
in the domination \eqref{eq:domBk} can be expanded into 
\begin{align*} \Espp\left[\ell(\mathrm{B}_k)^4 |\cf_{k-1}\right] +
4L_{k-1}^2\Espp\left[\ell(\mathrm{B}_k)^2|\cf_{k-1}\right] + 4
L_{k-1}\Espp\left[\ell(\mathrm{B}_k)^3|\cf_{k-1}\right]
\end{align*}
Then, recall the density \eqref{eq:DensLBk} of $\ell(\mathrm{B}_k)$
conditionally on $\cf_{k-1}$. In the proof of Lemma \ref{Momentsr}, we show that
for any $\lambda >0$, we have a.s.
\[
\Espp\left[\ell\left(\mathrm{B}_k\right)^\lambda|\cf_{k-1}\right] =\int
r_{L_{k-1}}(dx) x^\lambda \le C_1\cdot L_{k-1}^{-\lambda} + C_2 \cdot
L_{k-1}^{-\lambda-2}
\]
with $C_1$ and $C_2$ deterministic constants.
Thus, $\Espp[(L_k^2-L_{k-1}^2)^2|\cf_{k-1}]$ is a.s. bounded by $F(L_{k-1})$,
where $F$ is a nonincreasing bounded nonnegative function. In the end, we get 
\begin{align*}
 \Espp\left[H_k^2 \ind{E_{k-1}} \right] & \le \Espp\Bigg[
\Bigg(\frac{C}{L_{k-1}^2}
\int_{\dt_{k-1}^*} \theta(s) \frac{\ell(ds)}{L_{k-1}}  \\
 \rightalign{+ \int_0^\infty \frac{\expp{-u}\
du}{2L_{k-1}^2}\sqrt{1/2+u/L_{k-1}^2}
h_{\emptyset,k} \Bigg)^2 F(L_{k-1}) \ind{E_{k-1}} \Bigg]}
& \le F(k^{2-4\epsilon}) \Bigg(C\cdot k^{-2+4\epsilon} \Espp\left[ \left(
\int_{\dt_{k-1}^*} \theta(s) \frac{\ell(ds)}{L_{k-1}} \right)^2  \right]
\\
 \rightalign{ + C'\cdot k^{-2+4\epsilon} \left(\int_0^\infty \expp{-u}
\sqrt{1/2+u/k^{1-2\epsilon}}\right)^2 \Espp[h_{\emptyset,k}^2]\Bigg). } 
\end{align*}
Hence, using the fact that $\int_{\dt_{k-1}^*} \theta(s) \ell(ds)/L_{k-1}$ is
bounded in $L^2$ (which is precisely Lemma \ref{lem:BorneL2}), we find that
$\Espp[H_k^2
\ind{E_{k-1}}]=O(k^{-2+4\epsilon})$. Putting this together with the estimate on
$\Espp[G_k^2 \ind{E_{k-1}}]$, we get that $\Espp[(M_k-M_{k-1})^2
\ind{E_{k-1}}]=O(k^{-7/4+7/2\epsilon})$. If $\epsilon<1/14$,
\[ \Espp[(M_k-M_{k-1})^2 \ind{E_{k-1}}] = O(k^{-7/4+7/2\epsilon}) =
o(k^{-3/2}).
\]
Hence, we get
\[
\lim_{n\to\infty} n^{1/2} \sum_{k=n}^\infty \Espp\left[(M_k-M_{k-1})^2
\ind{E_{k-1}}\right]=0. 
\]
This shows that the random sequence $n^{1/4} \sum_{k=n}^\infty (M_k-M_{k-1})
\ind{E_{k-1}}$ converges to 0 in $L^2$, hence in probability. But, since 
there a.s. exists $k_0\ge 1$ such that $\ind{E_{k}}=1$ for all $k\ge k_0$, the
sequence $n^{1/4} \sum_{k=n}^\infty (M_k-M_{k-1})$ also converges to 0 in
probability, which is what we wanted to prove.
\end{proof}

\section{Proof of the main theorem}

We can now turn to the proof of the actual convergence towards a nontrivial
limit, in the asymptotic $n^{1/4}$. The main idea is to apply the Martingale
Central Limit Theorem (Corollary 3.1 in \cite{Hall1980}) to 
\[ M_n^* = X_n^* -\int_{\dt_n^*} \theta(s)\ \ell(ds). \]
We recall this theorem below for convenience:
\begin{nontheorem}[Hall, Heyde \cite{Hall1980}]
 Let $(M_n,\ n\ge 1)$ be a zero-mean square-integrable $(\cg_n)$-martingale,
and let $\eta^2$ be an a.s. finite random variable. Suppose that, for some
sequence $a_n$ increasing to $+\infty$, we have
\begin{enumerate}
 \item \emph{(Asymptotic smallness)} For all $\epsilon >0$, we have the
convergence in probability
\[ 
 \lim_{n\to \infty} a_n^{-2} \sum_{k=1}^n \Esp \left[\left.(M_k-M_{k-1})^2
\ind{\{ |M_k-M_{k-1}|> \epsilon a_k \}} \right| \cg_{k-1}\right] = 0 
\]
 \item \emph{(Convergence of the conditional variance)} We have the convergence
in probability
\[
 \lim_{n\to\infty} a_n^{-2} \sum_{k=1}^n \Esp\left[(M_k-M_{k-1})^2 |
\cg_{k-1}\right] = \eta^2.
\]
\end{enumerate}
 Then, the sequence $(a_n^{-1}M_n,\ n\ge 1)$ converges in distribution to a
random variable $Z$ with characteristic function $\Esp[\exp(-\eta^2t^2/2)]$.
\end{nontheorem}

However, $M_n^*$ is not a martingale in the filtration $(\cf_n,\ n\ge 1)$,
because the $(n+1)$st branch $\cb_{n+1}$ might be connected to $\dt_n$ through
a vertex on $\llbracket \emptyset,\xrac \rrbracket$. In that case,
$M_{n+1}^*-M_n^*$ has a nonnegative $\cf_n$-measurable part, corresponding to
the atoms on $\llbracket s_{\emptyset,n+1},\xrac \rrbracket$. For this
reason, we will consider 
\[ \widehat{M}_n = \sum_{s\in \dt_n\setminus \dt_1} \ind{\{\theta(s-) >
\theta(s)\}}
-\int_{\dt_n\setminus \dt_1} \theta(s)\ \ell(ds),\quad n\ge 2 \]
and $\widehat{M}_1=0$. The process $(\widehat{M}_n,n\ge 1)$ is a
$(\cf_n)$-martingale. It is actually more convenient to introduce
the filtration $(\cg_n,\ n\ge 1)$, defined by:
\[ \cg_n = \sigma(\{(\dt_m,\ m\ge 1),\ (\theta(s),\ s\in \dt_n)\}),
\]
Notice that the branching point $s_{n+1}=\cb_{n+1}\cap \dt_n$,
as well as $\ell(\cb_{n+1})$ and $\theta(s_{n+1})$ are all $\cg_n$-measurable.
In this filtration, $\widehat{M}$ is also a martingale. Indeed, it is obvious
that $\widehat{M}$ is $\cg$-adapted. Furthermore, we have
\[ \widehat{M}_{n+1}-\widehat{M}_n= \sum_{s\in \cb_{n+1}}
\ind{\{\theta(s-)>\theta(s)\}}-\int_{\cb_{n+1}} \theta(s)\ \ell(ds),\]
 which is, conditionally on $\cg_n$, distributed as $N_{\ell(\cb_{n+1})}$, where
$N$ is the martingale from \eqref{MartRec1} for a linear record process started
at $\theta(s_{n+1})$ . Thus, $\Espp[\widehat{M}_{n+1}-\widehat{M}_n|\cg_n] =
0$. 

\subsection{Convergence of the asymptotic variance}

 In order to get a convergence in distribution of $n^{-1/4} \widehat{M}_n$, we
first need to compute the asymptotic variance of the martingale. This is done in
the following proposition.
\begin{proposition} \label{Prop:ConvVarCond}
 We have:
\begin{equation} \label{ConvVarCond}
 \lim_{n\to\infty} \frac{1}{\sqrt{n}} \sum_{k=2}^n
\Espp \left[\left(\widehat{M}_{k}-\widehat{M}_{k-1}\right)^2 \Big|
\cg_{k-1}\right] = \sqrt{2}\Theta, 
\end{equation}
in probability. 
\end{proposition}
\begin{proof}
 Using the martingale from \eqref{MartRec2}, in the present case of a linear
record process started at $\theta(s_{k})$, we easily get that, for
$k\ge 2$, 
\begin{equation}\label{7*}
 \Espp \left[\left(\widehat{M}_{k}-\widehat{M}_{k-1}\right)^2 \Big|
\cg_{k-1}\right] = \Espp\left[ \int_{\cb_{k}} \theta(s)\ \ell(ds) \Big|\cg_{k-1}
\right].
\end{equation}
A Law of Large Numbers argument will show that we have
 \begin{equation}\label{LGNMart} \lim_{n\to\infty} \frac{1}{\sqrt n}
\sum_{k=2}^n
\Espp\left[\int_{\cb_{k}} \theta(s)\ \ell(ds) \Big|\cg_{k-1}\right] =
\lim_{n\to\infty} \frac{1}{\sqrt n} \int_{\dt_n^*\setminus
\bra{\xrac}{\lf_1}} \theta(s)\ \ell(ds). \end{equation}
We postpone the proof of this equality to the end of this section. Now, recall
Proposition 6.3 in \cite{Abraham2011}, which shows that a.s.
\[ \lim_{n\to\infty} \frac{1}{\sqrt n} \int_{\dt_n^*} \theta(s)\ \ell(ds) =
\sqrt{2}\Theta. \]
Since $\dt_n\setminus \cb_1=\dt_n^*\setminus \bra{\xrac}{\lf_1}$, the
convergence \eqref{ConvVarCond} will follow if we manage to prove that 
\[ S_n = \frac{1}{\sqrt{n}}\int_{\llbracket
\xrac,\mathrm{x}_1\rrbracket}\theta(s)\ \ell(ds) \]
converges in probability to 0. We will simply compute the first
moment:
\begin{align*}
 \sqrt{n} \Espp[S_n] = \Espp\left[ \int_{\llbracket
\xrac,\mathrm{x}_1\rrbracket} \theta(s)\ \ell(ds) \right] & =
\Espp \left[ \int_{h_{\emptyset,n}}^{L_1} \theta(s)\  ds \right] \\
 & = \Espp\left[ \int_0^{L_1-h_{\emptyset,n}}
\Esp_{\trac}[\theta(s)]\ ds\right],
\intertext{by the Markov property of
$\theta$ at $h_{\emptyset,n}$. We can compute this expectation using
\eqref{EspTheta}:}
 & = \Espp\left[ \int_0^{L_1-h_{\emptyset,n}}
\frac{1-\expp{-s\trac}}{s}\ ds \right] \\
 & \le \Espp \left[ \int_0^{L_1} \frac{1}{s}
(s\trac)^{1/4}\ ds \right] = 4 \Espp\left[\trac^{1/4}
L_1^{1/4}\right], \end{align*}
by the elementary inequality $1-\exp(-t)\le t^{1/4}$. The Cauchy-Schwarz
inequality then gives the bound 
\begin{equation}\label{BorneSpine} \sqrt{n} \Espp[S_n] \le C\cdot
\Espp\left[\trac^{1/2}\right]^{1/2}. \end{equation}
 As $\trac$ is, conditionally on $\ct$, exponentially distributed with parameter
$h_{\emptyset,n}$, we get 
\[
 \Espp[S_n] \le C\cdot n^{-1/2} \Espp[\hrac^{-1/2}]^{1/2},
\]
which converges to 0 as $n\to\infty$ by \eqref{eq:MomHn}, which shows
\eqref{ConvVarCond}. 

We still have to show \eqref{LGNMart} to end the proof. The process
\begin{equation}\label{10*}
\left(Q_n = \sum_{k=2}^n \int_{\cb_k} \theta(s)
\ell(ds)-\Espp\left[\left.\int_{\cb_k}
\theta(s)\ell(ds)\right|\cg_{k-1}\right],\ n\ge1\right) 
\end{equation}
is a $\cg$-martingale. We will write 
\begin{equation} \label{8*}
\langle Q \rangle_n = \sum_{k=1}^n \Espp \left[
\left(\int_{\cb_k}
\theta(s)\ \ell(ds)\right)^2 \Bigg| \cg_{k-1}\right] - \Espp\left[\int_{\cb_k}
\theta(s)\ \ell(ds) \Big| \cg_{k-1}\right] ^2 
\end{equation}
for its quadratic variation process. Conditionally on $\cg_{k-1}$, the process
$(\theta(s),\ s\in \cb_{k})$ is distributed as a linear record process started
from $\theta(s_k)$. Hence, using \eqref{eq:MomHn} and \eqref{EspTheta}, we get:
\begin{equation}
\Espp\left[\int_{\cb_k} \theta(s)\ \ell(ds)\Big|\cg_{k-1}\right]  =
\Esp_{\theta(s_k)}\left[ \int_0^{\ell(\cb_{k})} \theta(s)\ ds \right]  =
\int_0^{\theta(s_k) \ell(\cb_k)} \frac{1-\expp{-u}}{u}\ du.\label{CrochQn1}
\end{equation}
Similarly, we have:
\begin{align*}
 \Espp\left[\left(\int_{\cb_k} \theta(s)\
\ell(ds)\right)^2\Bigg|\cg_{k-1}\right] & =
\Esp_{\theta(s_k)}\left[\left(\int_0^{\ell(\cb_k)} \theta(s)\ ds
\right)^2 \right] \\
 & = 2\cdot \Esp_{\theta(s_k)}\left[ \int_0^{\ell(\cb_k)} du \int_0^u dv\
\theta(u)\theta(v) \right].
\end{align*}
The latter can be computed by applying the Markov property at $u$, as well as
\eqref{EspTheta}, giving
\begin{multline}\label{CrochQn2}
 \Espp\left[\left(\int_{\cb_k} \theta(s)\ \ell(ds)
\right)^2\Bigg|\cg_{k-1}\right]=
\frac{1}{\theta(s_k)} \int_0^{\theta(s_k)\ell(\cb_k)}
\frac{1-\expp{-s}}{s}-\expp{-s}\ 
ds \\ + 2 \int_0^{\theta(s_k) \ell(\cb_k)} ds \int_0^s
dt \frac{1}{s-t}\left(\frac{1-\expp{-t}}{t}-\frac{1-\expp{-s}}{s}\right).
\end{multline}
Now, putting \eqref{CrochQn1} and \eqref{CrochQn2} together, compensations
occur, so that we get, after tedious computations:
\begin{eqnarray*}
 \langle Q \rangle_n & = & \sum_{k=1}^n \Espp\left[\left(\int_{\cb_k}
\theta(s)\ \ell(ds)\right)^2 \Bigg| \cg_{k-1}\right] - \Espp\left[\int_{\cb_k}
\theta(s)\ \ell(ds) \Big| \cg_{k-1}\right]^2 \\
 & = & \sum_{k=1}^n \frac{2}{\theta(s_k)}
\int_0^{\theta(s_k)\ell(\cb_k)} \frac{1-\expp{-s}}{s} -\expp{-s}\ ds
\\
 & & \ + 2 \int_0^{\theta(s_k)\ell(\cb_k)} ds \int_0^s
dt\ \frac{s\expp{-s}-t\expp{-t} -(s-t)\expp{-(s+t)}}{st(s-t)}.
\end{eqnarray*}
The term $s\expp{-s}-t\expp{-t} -(s-t)\expp{-(s+t)}$ being negative for $t<s$,
we get
\begin{eqnarray*}
 0  \le \langle Q \rangle_n & \le & \sum_{k=1}^n
\frac{2}{\theta(s_k)}
\int_0^{\theta(s_k)\ell(\cb_k)} \frac{1-\expp{-s}}{s} -\expp{-s}\ ds
\\ 
 & \le & \sum_{k=1}^n\frac{2}{\theta(s_k)}
\theta(s_k)\ell(\cb_k) = 2 \sum_{k=1}^n \ell(\cb_k),
\end{eqnarray*}
the second inequality coming from $(1-\expp{-s})/s-\expp{-s} \le 1$ if $s>0$.
Then, recall that by definition, $\sum_{k=1}^n \ell(\cb_k) \le L_n$, and
that $L_n$ is the square root
of a \textsl{Gamma}$(n,1)$-distributed variable (Proposition 5.2 in
\cite{Abraham2011}).
Thus, for any $\gamma>1/2$, we have 
\begin{equation}\label{9*}
\frac{1}{n^{\gamma}} \Espp[\langle Q \rangle_n] \le \frac{2}{n^{\gamma}}
\Espp[L_n] \rightarrow 0 
\end{equation}
Then, by the conditional Law of Large Numbers (Theorem 1.3.17
in \cite{Duflo}), we get that $n^{-1/4-\epsilon}Q_n$ converges a.s. to 0 for
any $\epsilon>0$, which implies \eqref{LGNMart}, hence ends the proof.
\end{proof}

\subsection{Asymptotic smallness}

We now turn to the proof of the asymptotic smallness of the sequence
$(\widehat{M}_n,\ n \ge 1)$. In order to prove this, we will use a
Liapounov-type criterion, which is sufficient to prove asymptotic
negligibility. 
\begin{proposition}\label{Prop:AsymNegl}
 We have the following convergence in probability:
\[ \lim_{n\to\infty} \frac{1}{\sqrt{n}} \sum_{k=1}^n \Espp\left[ (\widehat{M}_k
- \widehat{M}_{k-1})^2 \ind{\{|\widehat M_k-\widehat M_{k-1}|>\epsilon
n^{1/4}\}} \Big| \cg_{k-1} \right] = 0. \]
\end{proposition}
\begin{proof}
 We use the standard inequality $\ind{\{|\widehat M_k-\widehat M_{k-1}|>\epsilon
n^{1/4}\}} \le (\widehat M_k -\widehat M_{k-1})^2/\epsilon^2 \sqrt{n}$ to get
that, for $\epsilon >0$:
\begin{equation*}
 \frac{1}{\sqrt{n}} \sum_{k=1}^n \Espp\left[ \left(\widehat{M}_k -
\widehat{M}_{k-1}\right)^2 \ind{\{|\widehat
M_k-\widehat M_{k-1}|>\epsilon n^{1/4}\}} \Big| \cg_{k-1} \right] \le
\frac{1}{\epsilon^2 n}
\sum_{k=1}^n \Espp\left[\left(\widehat M_k-\widehat M_{k-1}\right)^4 \Bigg|
\cg_{k-1}\right] .
\end{equation*}
Using the martingale from \eqref{MartRec3}, we find that:
\begin{multline*}
 \frac{1}{\epsilon^2 n} \sum_{k=2}^n \Espp\left[\left(\widehat M_k -\widehat
M_{k-1}\right)^4 \Bigg| \cg_{k-1}\right] =
\frac{3}{\epsilon^2 n} \sum_{k=2}^n \Espp\left[\left(\int_{\cb_k} \theta(s)\
\ell(ds)\right)^2 \Bigg|\cg_{k-1}\right] \\
+ \frac{1}{\epsilon^2 n} \sum_{k=2}^n \Espp\left[\int_{\cb_k} \theta(s)\
\ell(ds) \Big|\cg_{k-1}\right].
\end{multline*}
In this expression, the term $n^{-1}\sum_{k=2}^n \Es[\int_{\cb_k}
\theta(s)\ell(ds) |\cg_{k-1}]$ converges in probability to 0, according to
\eqref{7*} and Proposition \ref{Prop:ConvVarCond}. Furthermore, recall from
\eqref{8*} that 
\begin{equation*}
 \frac{3}{\epsilon^2 n} \sum_{k=1}^n \Espp\left[\left(\int_{\cb_k} \theta(s)\
\ell(ds)\right)^2 \Bigg|\cg_{k-1}\right] = \frac{3\langle
Q\rangle_n}{\epsilon^2n} +
\frac{3}{\epsilon^2 n} \sum_{k=1}^n \Espp\left[\int_{\cb_k} \theta(s)\ \ell(ds)
\Big|\cg_{k-1}\right]^2,
\end{equation*}
where $Q$ is the martingale defined in \eqref{10*}. The quadratic variation
process $\langle Q\rangle_n/n$ converges in probability to 0 by
\eqref{9*}. Also, applying Lemma \ref{Series} to $a_k=\Espp[\int_{\cb_k}
\theta(s)\ell(ds) |\cg_{k-1}]$, we find that 
\[
 \frac{1}{n} \sum_{k=1}^n \Espp\left[\int_{\cb_k} \theta(s)\ \ell(ds)
\Big|\cg_{k-1}\right]^2 = 0,
\]
which ends the proof. \end{proof}

Putting all the previous elements together, we can now prove Theorem
\ref{The:TCLPrincipal}. 

\begin{proof}[Proof of Theorem \ref{The:TCLPrincipal}]
 First, we write that 
\[ n^{1/4}\left(\frac{X_n^*}{\sqrt{2n}} - \Theta\right) = \frac{\widehat
M_n}{\sqrt{2}n^{1/4}} + \frac{M_n^* -\widehat M_n}{\sqrt{2}n^{1/4}} +
n^{1/4}\left(\frac{1}{\sqrt{2n}} \int_{\dt_n^*} \theta(s)\ \ell(ds) -
\Theta\right).
\]
The convergence in distribution of $n^{-1/4}\widehat{M}_n$ towards a
non-degenerate limit $Z$ is a consequence of the Martingale Central Limit
Theorem recalled at the beginning of this section with $a_n=n^{1/4}$, as well
as the two Propositions \ref{Prop:ConvVarCond} and \ref{Prop:AsymNegl}.
Furthermore, the limiting random variable $Z$ is indeed distributed as
announced:
\[
 \Espp\left[\expp{itZ}\right] = \Espp\left[\expp{-t^2 \sqrt{2}\Theta/2}\right].
\]
 The term $e_n=M_n^* -\widehat
M_n$ can be expressed as
\[ e_n=M_n^*-\widehat M_n=\sum_{s\in \llbracket
\xrac,\mathrm{x}_1 \rrbracket} \ind{\{\theta(s-)> \theta(s)\}}
-\int_{\llbracket \xrac,\mathrm{x}_1 \rrbracket} \theta(s)\ \ell(ds).
\]
Using the martingale \eqref{MartRec2} to compute its second moment, we get 
\[ 
\Espp\left[e_n^2\right] = \Espp\left[\int_{\llbracket
\xrac,\mathrm{x}_1 \rrbracket} \theta(s)\ \ell(ds)\right], 
\]
so that $n^{-1/4}(M_n^*-\widehat M_n)$ converges to 0 in $L^2$, hence in
distribution as $n\to\infty$, by the previously used bound
\eqref{BorneSpine}. Finally, Proposition \ref{prop:ConvReste} and Proposition
\ref{prop:ConvMartFerm} show that the term $((2n)^{-1/2} \int_{\dt_n^*}
\theta(s)\ \ell(ds) - \Theta)$ brings no contribution in the asymptotic
$n^{1/4}$. This ends the proof.
\end{proof}

\begin{remark}
 Note that, under our assumptions, since $\Theta>0,\ \Pr_\infty$-a.s., we can
actually prove that the convergence in distribution of $n^{-1/4}\widehat{M}_n$
is \emph{mixing} (see \cite{Aldous1978} for more details on mixing limit
theorems). This implies in particular that we can obtain a standard normal
limit by renormalizing by the random factor $V_n$, where $V_n^2$ is the
conditional variance
\[
 V_n^2 = \sum_{k=1}^n
\Espp \left[\left(\widehat{M}_{k}-\widehat{M}_{k-1}\right)^2 \Big|
\cg_{k-1}\right],
\]
instead of the deterministic renormalization $n^{1/4}$. Corollary 3.2 in
\cite{Hall1980} then shows that $V_n^{-1}\widehat{M}_n$ converges in
distribution to a standard $\mathcal{N}(0,1)$ random variable.
\end{remark}

\section*{Technical appendix}

 In this appendix, we will state and prove several lemmas that are used
throughout the paper. They are purely analytic in nature, and their proof is
elementary, so we gather them here, for the reader's convenience. First, we
prove some universal bounds on $F(q,t)=\Esp_q[(\int_0^t\theta(s) ds)^2]$.

\begin{lemma} \label{lem:Fql}
 There exists $C_1,C_2,C_3,C_4 >0$ such that 
\begin{gather}
 F(q,t) \le C_1 (qt)^{3/2}+ C_2 qt^2\label{Fql1} \\
 F(q,t) \le C_3 \log^2(qt)+C_4 q^{-1/2} t^{1/2}\label{Fql2}
\end{gather}
\end{lemma}
\begin{proof}
 First, we recall that, according to \eqref{CrochQn2}, 
\begin{align*}
 F(q,t) & = \Esp_q\left[\Big(\int_0^t \theta(s)\ ds\Big)^2\right] \\
 & =\frac{1}{q} \int_0^{qt} \frac{1-\expp{-s}}{s}-\expp{-s}\ 
ds +  \int_0^{qt} ds \int_0^s dt\
\frac{1}{s-t}\left(\frac{1-\expp{-t}}{t}-\frac{1-\expp{-s}}{s}\right) \\
 &:= \tilde{F}(q,t) + G(ql).
\end{align*}
 The two estimates \eqref{Fql1} and \eqref{Fql2} will come from an asymptotic
analysis of 
\[
 \tilde{F}(q,t) =\frac{1}{q} \int_0^{qt} \frac{1-\expp{-s}}{s}-\expp{-s}\ 
ds 
\]
and 
\[
 G(qt) = \int_0^{qt} ds \int_0^s dt\
\frac{1}{s-t}\left(\frac{1-\expp{-t}}{t}-\frac{1-\expp{-s}}{s}\right).
\]
Let us start with $\tilde{F}$. We have
\begin{align*}
 \tilde{F}(q,t)& = \frac{1}{q} \int_0^{qt} \frac{1-\expp{-s}}{s}-\expp{-s}\ 
ds \\
 & = \frac{1}{q} \left(\gamma +\log(qt) + \int_{qt}^\infty \frac{\expp{-t}}{t}\
dt + \expp{-qt} -1\right).
\end{align*}
It is elementary to check that the function $\gamma +\log(x) + \int_{x}^\infty
\frac{\expp{-t}}{t}\ dt + \expp{-x} -1$ is equivalent to $x^2/4$ when $x\to 0$,
and equivalent to $\log(x)=o(\sqrt{x})$ when $x\to \infty$. Since
$\sqrt{x}=o(x^2)$ in the neighbourhood of $+\infty$ and $x^2 = o(\sqrt{x})$ in
the neighbourhood of $0$, by continuity, we can find constants $C_2$ and $C_4$
such that $\tilde{F}(q,t) \le C_2 (qt)^2/q$ and such that $\tilde{F}(qt) \le C_4
(qt)^{1/2}/q$. 

Turning to the function $G$, we can write 
\begin{align*}
 G(x) & = \int_0^x ds \int_0^s dt \frac{1}{s-t}
\left(\frac{1-\expp{-t}}{t}-\frac{1-\expp{-s}}{s}\right) \\
 & = \int_0^1 du \int_0^u dv \frac{1}{u-v}
\left(\frac{1-\expp{-xv}}{v}-\frac{1-\expp{-xu}}{u}\right), \\
\intertext{so that}
 G'(x) & = \int_0^1 du \int_0^u dv \frac{1}{u-v} (\expp{-xv}-\expp{-xu}), \\
\intertext{and that}
 G''(x) & = \int_0^1 du \int_0^u dv \frac{1}{u-v} (u\expp{-xu} - v\expp{-xv}).
\end{align*}
Thus, we have $G(0)=G'(0)=0$ and $G''(0)=1$. Since $G$ is smooth, we get that
$G(x) \sim x^2/2$ when $x\to 0$. 
\par As far as the asymptotic $x\to \infty$ is concerned, we can express
$G'(x)$ in terms of the exponential integral\footnote{Note that this integral
is to be taken in the sense of Cauchy's principal value.} function $Ei(x) =
\int_{-\infty}^x \exp(t)/t\ dt$:
\begin{align*}
 G'(x) & = \int_0^1 du\int_0^u \frac{dv}{u-v} (\expp{-xv}-\expp{-xu}) \\
 & = \int_0^1 du \expp{-xu}\int_0^{xu} \frac{dv}{v}(\expp{v}-1) \\
 & = \int_0^1 du \expp{-xu}(Ei(xu) - \log(xu) -\gamma).
\end{align*}
When $x\to \infty$, we get
 \begin{align*}
  G'(x)& \sim \int_0^1 du \expp{-xu} Ei(xu) = \frac{1}{x} \int_0^x du \expp{-t}
Ei(t) \\
 & \sim \frac{\log x}{x}. 
 \end{align*}
Integrating from $0$ to $x$, we get $G(x)\sim \log^2x=o(\sqrt{x})$ when
$x\to\infty$. Again, $\sqrt(x)=o(x^2)$ in the neighbourhood of $+\infty$ and
$x^2=o(\sqrt{x})$ in the neighbourhood of 0, so that by continuity, there exist
two constants $C_1$ and $C_2$ such that $G(x) \le C_1 x^2$ and such that
$G(x)\le C_2 x^{1/2}$. Thus, we get the two dominations \eqref{Fql1} and
\eqref{Fql2}. \end{proof}

We now turn to a useful estimation of the moments of the distribution $r_a(dx)$
introduced in \eqref{eq:DensLBk}:
\[
 r_a(dx) = (a+x) \expp{-x^2/2-ax}\ind{(0,\infty)}(x)\ dx.
\]
\begin{lemma} \label{Momentsr}
 Let $\lambda >0$. Then, if $(a(n),\ n\ge 1)$ is some sequence in $\R_+$
increasing to $+\infty$, then,
as $n\to \infty$, we have $\int_0^\infty r_{a(n)}(dx) x^\lambda =
O(a(n)^{-\lambda})$.
\end{lemma}
\begin{proof}
 This is fairly easy: if $\lambda >0$, we can write
\begin{align*}
 \int_0^\infty r_{a(n)}(dx) x^{\lambda} & = \int_0^\infty x^\lambda
(a(n)+x)\expp{-x^2/2-a(n)x}\ dx \\
 & = \int_0^\infty \frac{u^{\lambda}}{a(n)^\lambda}\left(a(n) +
\frac{u}{a(n)}\right)
\expp{-u^2/(2a(n)^2)-u}\ \frac{du}{a(n)} \\
 & \le \frac{1}{a(n)^\lambda} \int_0^\infty u^\lambda \expp{-u}\ du +
\frac{1}{a(n)^{\lambda+2}} \int_0^\infty u^{\lambda+1} \expp{-u}\ du,
\end{align*}
which ends the proof.
\end{proof}

\begin{lemma} \label{lem:majoHqr}
 For any $0<q<\infty$ and any $v\ge 0$, we have
\[
 \Es_q^{(v)}\left[\Theta \right] \le
\sqrt{\pi/2} \min(qv,\sqrt{v}).
\]
\end{lemma}
\begin{proof}
 We will use formula (21) from \cite{Abraham2011}, stating that, in our context,
if $Y$ is a Rayleigh-distributed variable, then 
\begin{equation*}
 \Es_q^{(v)}\left[\Theta \right] = \sqrt{v}
\int_0^{q\sqrt{v}} \Es\left[\expp{-tY}\right]\ dt.
\end{equation*}
We simply expand the Laplace transform, giving
\begin{align*}
 \Es_q^{(v)}\left[\Theta \right] & = \sqrt{v}
\int_0^{q\sqrt{v}} \int_0^\infty x\expp{-x^2/2} \expp{-tx}\ dx\ dt \\
 & = \sqrt{v} \int_0^\infty \expp{-x^2/2} \left( 1-\expp{-xq\sqrt{v}}\right)\
dx.
\end{align*}
Now, we use the obvious inequality $1-\exp(-x) \le
\min(x,1)$, to get the desired domination, since $qv\int_0^\infty
x\exp(-x^2/2)=qv$ and $\sqrt{v}\int_0^\infty \expp{-x^2/2} dx = \sqrt{\pi
v/2}$.
\end{proof}
Finally, the next lemma is needed to prove the asymptotic smallness of the
martingale $\widehat{M}_n$.

\begin{lemma}\label{Series}
 Let $(a_n,n\ge 1)$ be a nonnegative sequence such that 
\[ 
 \lim_{n\to\infty} \frac{1}{\sqrt{n}} \sum_{k=1}^n a_k <\infty.
\]
Then, we have
\[
 \lim_{n\to\infty} \frac{1}{n} \sum_{k=1}^n a_k^2 = 0.
\]
\end{lemma}
\begin{proof}
Let $s_n= n^{-1/2}\sum_{k=1}^n a_k$. Taking the difference $s_n-s_{n-1}$, we
easily see that $n^{-1/2}a_n$ converges to 0. Then, if $\epsilon >0$, there
exists $n_0\ge 1$ such that for all $n\ge n_0$, $a_n < \epsilon\sqrt{n}$. Thus,
if $n\ge n_0$, we have 
\begin{align*}
 \sup_{k\le n} a_k & \le \sup_{k< n_0} a_k + \sup_{n_0 \le k\le n} a_k \\
 & \le \sup_{k<n_0} a_k+\epsilon,
\end{align*}
which proves that actually 
\[
 \lim_{n\to\infty} \frac{\sup_{k\le n} a_k}{\sqrt{n}} = 0.
\]
Then, we simply write
\[
 \frac{1}{n} \sum_{k=1}^n a_k^2 \le \left(\frac{\sup_{k\le n} a_k}{\sqrt{n}}
\right) \left( \frac{1}{\sqrt{n}} \sum_{k=1}^n a_k \right)
\]
to conclude.
\end{proof}

\begin{Ack}
The author wishes to thank his PhD advisers, Romain Abraham and Jean-François
Delmas, for a very careful reading of the manuscript, as well as for valuable
advice.
\end{Ack}

\end{document}